%% file: main.tex
\documentclass[11pt, reqno]{amsart}
\usepackage{mathrsfs}
\usepackage[active]{srcltx}
\usepackage{mathrsfs,amsmath}
\usepackage{mathtools}
\usepackage{longtable}
\usepackage{todonotes}
\usepackage{amssymb}
\usepackage{cite}
\allowdisplaybreaks

\usepackage{xcolor}
\definecolor{bluecite}{HTML}{0875b7}

\usepackage[unicode=true,
bookmarksopen={true},
pdffitwindow=true,
colorlinks=true,
linkcolor=bluecite,
citecolor=bluecite,
urlcolor=bluecite,
hyperfootnotes=false,
pdfstartview={FitH},
pdfpagemode= UseNone]{hyperref}

\newcommand{\ler}[1]{\left(#1\right)}
\newcommand{\dd}{\mathrm{d}}
\newcommand{\cT}{\mathcal{T}}

\newcommand{\be}{\begin{equation}}
\newcommand{\ee}{\end{equation}}

\newcommand{\potmu}{\cT_{\mu}^{(p)}}

\newcommand{\potphimu}{\cT_{\Phi(\mu)}^{(p)}}

\newcommand{\Rn}{\mathbb{R}^n}

\newcommand{\dwp}{d_{\mathcal{W}_p}}

\newcommand{\supp}{\mathrm{supp}}

\newcommand{\Wp}{\mathcal{W}_p}

\newcommand{\R}{\mathbb{R}}

\newcommand{\wprnn}{\mathcal{W}_p(\mathbb{R}^n,d_N)}
\newcommand{\wornn}{\mathcal{W}_1(\mathbb{R}^n,d_N)}
\usepackage[left=3cm,right=3cm,top=3.0cm,bottom=3.0cm]{geometry}
\usepackage{marginnote}
\usepackage[normalem]{ulem}

\newtheorem{theorem}{Theorem}[section]
\newtheorem{example}[theorem]{Example}
\newtheorem{proposition}[theorem]{Proposition}
\newtheorem{lemma}[theorem]{Lemma}
\newtheorem{corollary}[theorem]{Corollary}

\newtheorem{remark}[theorem]{Remark}
\numberwithin{equation}{section}

\usepackage{bbm}  
\usepackage{dsfont}


\title[Rigidity over $\Rn$ with smooth norms]{
 Wasserstein Rigidity over $\Rn$ with smooth norms}

\author[Zolt\'an M. Balogh]{Zolt\'an M. Balogh}
\address{Zolt\'an M. Balogh, Universit\"at Bern\\ Mathematisches Institut (MAI)\\ Sidlerstrasse 12\\ 3012 Bern\\ Schweiz}
\email{zoltan.balogh@unibe.ch}

\author[Eric Str\"oher]{Eric Str\"oher}
\address{Eric Str\"oher, Universit\"at Bern\\ Mathematisches Institut (MAI)\\ Sidlerstrasse 12\\ 3012 Bern\\ Schweiz}
\email{eric.stroeher@unibe.ch}

\author[Tam\'as Titkos]{Tam\'as Titkos}
\address{Tam\'as Titkos, Corvinus University of Budapest, Department of Mathematics \\
Fővám tér 13-15 \\ Budapest 1093 \\ Hungary\\ and HUN-REN Alfr\'ed R\'enyi Institute of Mathematics\\ Re\'altanoda u. 13-15.\\
Budapest 1053\\ Hungary}
\email{tamas.titkos@uni-corvinus.hu}

\author[D\'aniel Virosztek]{D\'aniel Virosztek}
\address{D\'aniel Virosztek, HUN-REN Alfr\'ed R\'enyi Institute of Mathematics\\ Re\'altanoda u. 13-15.\\Budapest H-1053\\ Hungary}
\email{virosztek.daniel@renyi.hu}

\thanks{Z. M. Balogh and E. Str\"oher are supported by the Swiss National Science Foundation, Grant Nr. {200020\_228012}.}

\thanks{T. Titkos is supported by the Hungarian National Research, Development and Innovation Office (NKFIH) under grant agreements no. K134944 and no. Excellence\_151232, and by the Momentum program of the Hungarian Academy of Sciences under grant agreement no. LP2021-15/2021.}

\thanks{D. Virosztek is supported by the Momentum program of the Hungarian Academy of Sciences under grant agreement no. LP2021-15/2021, by the Hungarian National Research, Development and Innovation Office (NKFIH) under grant agreement no. Excellence\_151232, and partially supported by the ERC Synergy Grant No. 810115.}

\begin{document}
	\begin{abstract}
We study $p-$Wasserstein spaces $ \Wp(\Rn, d_N)$  over $\Rn$ equipped with a norm metric $d_N$. We show that, if the norm is smooth enough, then the Wasserstein space is isometrically rigid whenever $p \neq 2$. We also show that, even when $p=2$, we can recover the isometric rigidity of the Wasserstein space  $\mathcal{W}_2(\Rn, d_N)$ when $N$ is an $l_q-$norm and $q>2$. 
 \end{abstract}
	\maketitle 
	\tableofcontents
	

\input{Norm_intro}

\input{Metric_char_Dirac}

\input{Rigid_upgrade_3}

\input{Rigid_continuous_C2}

	\end{document}

%% file: Norm_intro.tex
\section{Introduction and main results}
\label{sec:intro}
For a fixed $p\geq1$ and a complete separable metric space $(X, d_X)$, the $p$-Wasserstein space $\Wp(X, d_X)$ is the set of Borel probability measures with finite $p$-moment, equipped with the so-called Wasserstein  distance $d_{\Wp}$. More precisely, we say that a Borel probability measure $\mu$ is in the $p$-Wasserstein space if 
$$
\int_X d^p_X(x, x_0)d\mu(x)<\infty
$$
for some (and hence any) $x_0 \in X$, and the Wasserstein distance is given by 
$$
d_{\mathcal{W}_p}(\mu, \nu)=\inf_{\pi \in \Pi (\mu, \nu)}\left\lbrace \left( \iint_{X \times X} d_X^p(x, y)d\pi(x,y)\right)^{\frac{1}{p}} \right\rbrace, 
$$
where $\Pi(\mu, \nu)$ is the set of couplings between $\mu$ and $\nu$, i.e. the set of probability measures $\pi \in \mathcal{P}( X \times X)$ with $p_{1 \#}\pi=\mu$ and $p_{2 \#}\pi=\nu$, for the projection operators $p_1(x, y)=x$ and $p_2(x, y)=y$; here $p_{\#}$ denotes the push-forward operation, which we define below.

The Wasserstein space $\Wp(X, d_X)$ inherits a number of properties from its base space $(X, d_X)$, see \cite{AG, Figalli, V, Villani} for an overview. For example, if $(X, d_X)$ is complete and separable, then so is $\Wp(X, d_X)$. 

Another property of any Wasserstein space is that it contains an isometric copy of its base space. Indeed, the map $x \to \delta_x$ is an isometric embedding from the space $(X, d_X)$ to $\Wp(X, d_X)$, where $\delta_x$ is the Dirac mass located at $x \in X$. Furthermore, the convex hull of Dirac masses (i.e. the set of measures with finite support) is dense in $\Wp(X,d_X)$ (see e.g. \cite[Theorem 6.18]{Villani}). If $\phi$ is an isometry (i.e. a distance-preserving and bijective map) on $(X, d_X)$, it induces an isometry of $\Wp(X, d_X)$ by the push-forward operation. More precisely $\phi_{\#}: \Wp(X, d_X) \to \Wp(X, d_X)$ defined by $$\phi_{\#}(\mu)(A) = \mu (\phi^{-1}(A))\quad\mbox{for}\quad A\subset X$$ is an isometry of $\Wp (X, d_X)$. An isometry $\Phi$ of the Wasserstein space that can be written as $\Phi=\phi_{\#} $ is called a trivial isometry. If the isometry group of $\Wp (X, d_X)$ contains only trivial isometries, we say that the Wasserstein space is isometrically rigid. 

In \cite{K}, Kloeckner showed that for $p=2$, the space $\mathcal{W}_2(\Rn, d_E)$ (where $d_E$ is the usual Euclidean distance) is not rigid. That is, its isometry group contains non-trivial isometries, called shape preserving isometries when $n\geq 2$.
Another example of a non-rigid space was given in \cite{GTV1}, where Geh\'er, Titkos and Virosztek showed that $\mathcal{W}_1([0,1], |\cdot |)$ admits non-trivial, so-called mass splitting isometries.

In contrast, Bertrand and Kloeckner showed in \cite{BK, BK2} that if $(X, d_X)$ is Hadamard, i.e. a Riemannian manifold with negative sectional curvature, then $\mathcal{W}_2(X, d_X)$ is isometrically rigid. 
Furthermore Santos-Rodriguez proved in \cite{S-R} that the same is true when $(X, d_X)$ is a manifold with strictly positive sectional curvature.  Rigidity in the case of the subriemannian Heisenberg group was shown by Balogh, Titkos and Virosztek  in  \cite{BTV}; in \cite{BKTV}, the same authors together with Kiss showed that the space $\mathcal{W}_p(\R^2, d_{\max})$ is isometrically rigid for any $p \geq 1$, where $d_{\max}$ is the distance induced by the maximum norm on the plane. This shows that there is an abundance of metric spaces where the associated $p$-Wasserstein space is rigid. 

On the other hand, there is also an abundance of non-rigid Wasserstein spaces, as shown in the recent result by Che, Galaz-Garcia, Kerin and Santos-Rodriguez, who proved in \cite{S-R2} that for any Hilbert space $(H, d_H)$ and any proper metric space $(Y, d_Y)$,  $\mathcal{W}_2(H \times Y, d_{H\oplus_{2} Y})$ is not rigid.
In this paper, we study the isometric rigidity of the Wasserstein spaces $\Wp(X, d_X)$ for the class of general normed spaces of $(\Rn, d_N)$.

We recall that in \cite{GTV2}, Geh\'er, Titkos and Virosztek showed that, if $(\Rn, d_H)$ is a Hilbert space with $d_H$ the distance induced by a scalar product, then for any $p\geq1$ with $p\neq 2$, the space $\Wp(\Rn, d_H)$ is isometrically rigid. We generalize their result in the following way: instead of requiring that the norm comes from an inner product, we show that only the smoothness of the norm is enough to ensure the isometric rigidity of the Wasserstein space. Our first theorem goes as follows: 

\begin{theorem}
\label{thm1}
If $N: \Rn \to \R_{+}$ is a strictly convex norm that is $C^{2}$-smooth, then the Wasserstein space $\Wp(\Rn,d_N)$ is isometrically rigid for all $p\in [1, \infty)$, $p \neq 2$.
\end{theorem}
Here we say that the norm $N$ is $C^2$-smooth if at any point $x \in \Rn \setminus \lbrace 0 \rbrace$, $N$ is (at least) twice differentiable at $x$. 

Our proof will be done in three steps. We first prove a geometric characterization of Dirac masses, which will in particular imply that for any $x \in \Rn$ and any isometry of the Wasserstein space $\Phi$, there exists $y \in \Rn$ such that $\Phi(\delta_x)=\delta_y$. 
In our second step we show a dimension upgrading result, which says that if an isometry $\Phi$ acts trivially on measures supported on certain subspaces of $\Rn$, then $\Phi$ has to act trivially on measure supported on the whole $\Rn$.
In the third step we show that, using the $C^2$-smoothness of the norm, we can find a proper subspace $L$ such that, if a measure is supported on $L$, then so is its image by an isometry $\Phi$. Finally, we combine these steps and, with an induction argument on the dimension of $\Rn$, we can prove the isometric rigidity of the Wasserstein space $\mathcal{W}_p(\Rn, d_N)$.   

We cannot expect this result to hold in general for $p=2$. Indeed, as mentioned above Kloeckner \cite{K} and Che, Galaz-Garcia, Kerin and Santos-Rodriguez \cite{S-R2} gave examples of Wasserstein spaces over certain special types of normed spaces in $\Rn$ which allowed shape-preserving isometries. 

In our proof of Theorem \ref{thm1}, we only need the condition $p \neq 2$ for certain steps; we can thus adapt the proof when the norm $N$ is an $l_q$-norm to show isometric rigidity in the case $p=2$, where the $l_q$-norms are given by $N_q(x)=\left( \sum_{i=1}^n |x_i|^q \right)^{\frac{1}{q}}$. Specifically, we prove the following theorem:

\begin{theorem}
\label{thm2}
If $q >2$ and $d_q$ is the distance function induced by the $l_q$-norm, then the Wasserstein space $\mathcal{W}_2(\Rn, d_q)$ is isometrically rigid. 
\end{theorem}
This theorem only considers the case $q>2$; we will show that $\mathcal{W}_2(\Rn, d_q)$ is also isometrically rigid when $1 \leq q < 2$ in our upcoming paper \cite{STV}.

Our paper is structured in the following way: in the next section, we show the metric characterization of Dirac masses. In section \ref{sec:upgrading}, we prove the dimension upgrading proposition. In section \ref{sec:potential}, we will show Theorem \ref{thm1}, first when $p <2$, then for $p>2$; we finish the section by adapting the $p>2$-argument to show Theorem \ref{thm2}.

%% file: Metric_char_Dirac.tex
\section{Metric characterization of Dirac masses}
\label{sec:diracs}

When investigating the isometric rigidity  of Wasserstein spaces, a metric characterization of Dirac masses is very often a key tool. Indeed,  since isometries preserve distances, such a characterization shows that the image under an isometry of a Dirac mass is another Dirac mass.
If that is the case, then for an isometry $\Phi:\Wp(\mathbb{R}^n,d_N)\to\Wp(\mathbb{R}^n,d_N)$, the map $\psi:(\Rn, d_N)\to (\Rn, d_N)$ defined by $\Phi(\delta_x)=\delta_{\psi(x)}$ is an isometry of $\Rn$, and $(\Phi \circ \psi^{-1}_{\#})(\delta_x)=\delta_x$ for all $x \in \Rn$. 
 If we can show that $\Phi\circ\psi_{\#}^{-1}$ is the identity on $\Wp(\mathbb{R}^n,d_N)$, we can conclude that $\Phi=\psi_{\#}$ is a trivial isometry. Thus we will in the following sections sometimes assume that $\Phi(\delta_x)=\delta_x$, since this can be obtained by composition with the trivial isometry $\psi_{\#}^{-1}$.

In the special case where $N$ is a Hilbert norm, a metric characterization for Dirac mass has been found in \cite[Lemma 3.5]{GTV2}. As we will see, the same characterization holds for a wide class of norms.
In what follows, we say that a triple $(\mu,\nu,\eta)$ of measures in $\Wp(\Rn, d_N)$ is $\Wp$-aligned if $\mu, \nu$ and $\eta$ are distinct and
\begin{equation}\label{eq:aligned}\dwp(\mu, \nu)+\dwp(\nu, \eta)=\dwp(\mu, \eta).\end{equation}

In this section, we prove two statements for the characterization of Dirac masses; Proposition \ref{pbigger1} when $p>1$ for an arbitrary norm $N$, and Proposition \ref{prop:pequal1} when $p=1$ for a strictly convex norm $N$. We start by looking at the Wasserstein space $\mathcal{W}_p(\mathbb{R}^n, d_N)$ with $p>1$. Then the following proposition holds:
\begin{proposition}
\label{pbigger1}
Let $p>1$, and consider the $p$-Wasserstein space $\Wp(\mathbb{R}^n,d_N)$, where $N:\Rn \to \R_{+}$ is a norm on $\Rn$. For a measure $\mu \in\wprnn$, the following are equivalent:
\begin{enumerate}
\item $\mu$ is a Dirac mass, i.e. there exists an $x \in \mathbb{R}^n$ such that $\mu = \delta_x$.
\item For all $\nu \in \wprnn$, $\mu \neq \nu$, there exists an $\eta \in\wprnn$ such that the triple $(\mu,\nu,\eta)$ is $\mathcal{W}_p$-aligned. 
\end{enumerate}
\end{proposition}
\begin{proof}
We start by showing that $(1) \Rightarrow (2)$. Take $x\in \mathbb{R}^n$, set $\mu=\delta_x$ and let $\nu \in \wprnn$ with $\mu \neq \nu$. Consider the dilation map $D_x:\mathbb{R}^n \rightarrow \mathbb{R}^n$ given by $D_x(y)=x+2(y-x)$. We claim that, with the measure $\eta:= (D_x)_{\#}\nu$, $(\mu, \nu, \eta)$ is $\mathcal{W}_p-$aligned. Indeed, we have
\begin{equation*}
\begin{split}
d_{\mathcal{W}_p}(\mu, \eta)&=\min_{\pi \in \Pi(\mu, \eta)}\left( \iint_{\Rn \times \Rn} d^p_N(\tilde{x},z)d\pi(\tilde{x},z)\right)^{1/p}\\
&=\left(\int_{\mathbb{R}^n} d_N^p(x,z)d\eta(z)\right)^{1/p}=\left(\int_{\mathbb{R}^n} d^p_N(x,D_x(y))d\nu(y)\right)^{1/p} \\
&=\left(\int_{\mathbb{R}^n} N(x-x-2(y-x))^pd\nu(y)\right)^{1/p}=\left(\int_{\mathbb{R}^n} 2^pN(y-x)^pd\nu(y)\right)^{1/p}  \\
&=2\left(\int_{\mathbb{R}^n} d^p_N(x,y)d\nu(y)\right)^{1/p}=2d_{\mathcal{W}_p}(\mu, \nu).
\end{split}
\end{equation*}
So we get that $d_{\mathcal{W}_p}(\mu, \nu)=\frac{1}{2}d_{\mathcal{W}_p}(\mu, \eta)$. 

By the triangle inequality, we have $d_{\mathcal{W}_p}(\mu, \eta) \leq d_{\mathcal{W}_p}(\mu, \nu) + d_{\mathcal{W}_p}(\nu, \eta)$, giving the bound $\frac{1}{2}d_{\mathcal{W}_p}(\mu, \eta) \leq d_{\mathcal{W}_p}(\nu, \eta)$. On the other hand, since $D_x$ is a transport map from $\nu$ to $\eta$, 
\begin{align*}
d_{\mathcal{W}_p}(\nu, \eta) \leq \left( \int_{\mathbb{R}^n} d^p_N(y, D_x(y))d\nu (y)\right)^{1/p}=\left( \int_{\mathbb{R}^n} N(y-x-2(y-x))^pd\nu(y)\right)^{1/p} \\
=\left( \int_{\mathbb{R}^n} N(y-x)^pd\nu(y)\right)^{1/p}=\left( \int_{\mathbb{R}^n} d^p_N(x,y)d\nu(y)\right)^{1/p} = d_{\mathcal{W}_p}(\mu, \nu)=\frac{1}{2}d_{\mathcal{W}_p}(\mu, \eta)
\end{align*}
showing that $\frac{1}{2}d_{\mathcal{W}_p}(\mu, \eta) \geq d_{\mathcal{W}_p}(\nu, \eta)$. This proves that \eqref{eq:aligned} holds, and thus $(1) \Rightarrow (2)$. 

We now show that $(2) \Rightarrow (1)$. Assume for contradiction that the measure $\mu$ is not a Dirac mass, and thus has two distinct points in its support, $x_1$ and $x_2$. For a point $y \in \Rn$ such that $d_N(x_1,y)>d_N(x_2,y)>0$,  we set $\nu=\delta_y$ to be the Dirac mass located at $y$. By the assumption there exists a measure $\eta$ ($=\eta_y$, as the measure depends on the choice of $y$) such that \eqref{eq:aligned} holds. Then 
\begin{align*}
d_{\mathcal{W}_p}(\mu, \eta)
&\leq \left(\iint_{\Rn \times \Rn} d_N^p(x,z)d(\mu \times \eta)(x,z) \right)^{1/p} \\
&\leq \left(\iint_{\Rn \times \Rn} (d_N(x,y) + d_N(y,z))^p d(\mu \times \eta)(x,z) \right)^{1/p} \\
&\leq\left(\iint_{\Rn \times \Rn} d_N^p(x,y) d(\mu \times \eta)(x,z) \right)^{1/p} + \left( \iint_{\Rn \times \Rn} d_N^p(y,z) d(\mu \times \eta)(x,z) \right)^{1/p} \\
&=\left( \int_{\mathbb{R}^n} d_N^p(x,y)d\mu(x) \right)^{1/p} + \left(\int_{\mathbb{R}^n} d_N^p(y,z) d\eta(z) \right)^{1/p}\\
&= d_{\mathcal{W}_p}(\mu, \nu) + d_{\mathcal{W}_p}(\nu, \eta),
\end{align*}
where in the third inequality, we used the Minkowski inequality. 
Since $(\mu, \nu, \eta)$ are $\mathcal{W}_p-$aligned, these inequalities are saturated. Since $p>1$, the Minkowski inequality becomes an equality only when there is a constant $\lambda$ such that  
\begin{equation}
\label{eq:Minkowski}
d_N(x,y)=\lambda d_N(y,z) \text{ for } (\mu \times \eta)\text{-a.e.} (x,z).
\end{equation} 

Then, by the definition of the support of a measure, for all $\varepsilon >0$, the balls $B_N(x_1, \varepsilon)$ and $B_N(x_2, \varepsilon)$ have positive $\mu$-measure and thus we can find $z' \in \supp (\eta)$, $x_1' \in B_N(x_1, \varepsilon)$ and $x_2' \in B_N(x_2,\varepsilon)$ such that $d_N(x_1', y)=\lambda d_N(y,z')$ and $d_N(x_2', y)=\lambda d_N(y,z')$. By our choice of $y$, if $\varepsilon$ is small enough we can guarantee that $d_N(x_1',y)>d_N(x_2',y)>0$.  But then $d_N(x_1',y)=\lambda d_N(y,z')=d_N(x_2',y)$, which is the desired contradiction. Therefore $\mu$ is indeed a Dirac mass.

\end{proof}
As the above proposition offers a metric characterization of Dirac masses, we have the following important corollary.
\begin{corollary}\label{corr:p>1 Dirac fix}
Let $p>1$, and consider the metric space $(\mathbb{R}^n,d_N)$, with $d_N$ a distance induced by a norm $N$. Assume that $\Phi:\Wp(\Rn,d_N)\to\Wp(\Rn,d_N)$ is an isometry. Then  there exists an isometry $\psi:(\mathbb{R}^n,d_N)\to(\mathbb{R}^n,d_N)$ such that $(\Phi\circ\psi_{\#}^{-1})(\delta_x)=\delta_x$ for all $x\in\Rn$.
\end{corollary}
\begin{proof}
For $x \in \Rn$, set $\mu=\delta_x$ the Dirac mass supported on $x$. We first show that $\Phi(\mu)$ is also a Dirac mass. For this, consider a measure $\nu' \in \Wp(\Rn, d_N)$, $\nu'\neq \Phi(\mu)$. Then, since $\mu$ is a Dirac mass, there exists a measure $\eta \in \Wp(\Rn, d_N)$ with $\eta$ different to $\mu$ and $\Phi^{-1}(\nu')$ such that $(\mu, \Phi^{-1}(\nu'), \eta)$ is $\Wp$-aligned. Since distances are preserved under isometry, we have that $(\Phi(\mu), \nu', \Phi(\eta))$ is also $\Wp$-aligned. Since this is true for any $\nu' \neq \Phi(\mu)$, by Proposition \ref{pbigger1}, $\Phi(\mu)$ is a Dirac mass, i.e. there exists $y \in \Rn$ such that $\Phi(\delta_x)=\delta_y$.  

To finish the proof, we define the map $\psi : \Rn \to \Rn$ by the relation $\delta_{\psi(x)}=\Phi(\delta_x)$. Since $\Phi$ is an isometry of $\mathcal{W}_p(\Rn, d_N)$, it is easy to see that $\psi$ is an isometry of $(\Rn, d_N)$, and that $\Phi \circ \psi_{\#}^{-1}(\delta_x)=\delta_x$. 
\end{proof}

When $p=1$, the characterization of Dirac masses from Proposition \ref{pbigger1} does not hold in general, as the following example shows. 
\begin{example} Consider the $l_1$-norm on $\R^2$ given by $N_1(x_1, x_2)=|x_1| + |x_2|$. Then there exists a measure $\mu$ which is not a Dirac mass such that, for any measure $\nu\neq \mu$, there exists a measure $\eta$ such that $(\mu, \nu, \eta)$ is $\mathcal{W}_1$-aligned. 
\end{example}
To see this, we take the measures 
$$\mu=\frac{1}{2}\delta_{(0,0)}+\frac{1}{2}\delta_{(1,0)} \in \mathcal{W}_1(\R^2, d_1) \text{ and } \nu=\delta_y \text{ for }y=(y_1,y_2)\in \R^2 \text {, }y_2 > 0.$$ Set $t_0=d_{\mathcal{W}_1}(\mu, \nu)$, $z=y + t_0 e_2$, and $\eta=\delta_z$. Then
\begin{align*}
d_{\mathcal{W}_1}(\mu, \eta)&=\frac{1}{2}d_1((0,0), (y_1, y_2 + t_0)) +\frac{1}{2} d_1((1,0), (y_1, y_2 + t_0)) \\
&=\frac{1}{2}(|y_1| + |y_2+t_0| + |y_1-1| + |y_2+t_0|) \\
&=\frac{1}{2}(2|t_0| + 2|y_2| + |y_1| + |y_1-1|)=t_0 + d_{\mathcal{W}_1}(\mu, \nu), 
\end{align*}
and thus the triple $(\mu, \nu,\eta)$ is $\mathcal{W}_1$-aligned, even though $\mu$ is not a Dirac mass.  A similar construction shows that for any measure $\nu$ we can find $\eta$ such that $(\mu, \nu, \eta)$ is $\mathcal{W}_1$-aligned.

However, for $p=1$, we can recover the statement of Proposition \ref{pbigger1} if we require the norm $N$ to be strictly convex.

\begin{proposition}
\label{prop:pequal1}
Let $p=1$, and consider the metric space $(\mathbb{R}^n,d_N)$, with $d_N$ induced by a strictly convex norm $N: \Rn \to \R_{+}$. Let $\mu \in \wornn$ be a measure. Then the following are equivalent:
\begin{enumerate}
\item There exists $x \in \mathbb{R}^n$ such that $\mu = \delta_x$ is a Dirac mass.
\item For all $\nu \in \mathcal{W}_1(\mathbb{R}^n, d_N)$, $\mu \neq \nu$, there exists an $\eta \in \mathcal{W}_1(\mathbb{R}^n, d_N)$ such that the triple $(\mu, \nu, \eta)$ is $\mathcal{W}_1$-aligned.  
\end{enumerate}
\end{proposition}
\begin{proof}

The proof of $(1) \Rightarrow (2)$ comes from the same calculation as in the proof of Proposition \ref{pbigger1}.
We thus only need to prove $(2) \Rightarrow (1)$. We assume for contradiction that $\mu$ is not a Dirac mass, and thus has two distinct points in its support, $x_1$ and $x_2$. We set $\nu=\delta_y$, where $y=\frac{1}{2}(x_1 + x_2)$,  
and consider the measure $\eta$ $(=\eta_y)$ such that \eqref{eq:aligned} holds. 
Then, similarly to the previous proof, 
\begin{align*}
d_{\mathcal{W}_1}(\mu, \eta)  &\leq  \iint_{\Rn \times \Rn} d_N(x,z)d(\mu \times \eta)(x,z) \\
&\leq \iint_{\Rn \times \Rn} (d_N(x,y) + d_N(y,z)) d(\mu \times \eta)(x,z)  \\
&\leq \iint_{\Rn \times \Rn} d_N(x,y) d(\mu \times \eta)(x,z)  + \iint_{\Rn \times \Rn} d_N(y,z) d(\mu \times \eta)(x,z)  \\
&= \int_{\mathbb{R}^n} d_N(x,y)d\mu(x)  +  \int_{\mathbb{R}^n}  d_N(y,z) d\eta(z)  = d_{\mathcal{W}_1}(\mu, \nu) + d_{\mathcal{W}_1}(\nu, \eta)
\end{align*}

In this case, instead of using the Minkowski inequality, we simply used the triangle inequality. Since the triple $(\mu, \nu, \eta)$ is $\mathcal{W}_1-$aligned, the inequalities are saturated, and we get that 
\begin{equation}
d_N(x,z)=d_N(x, y) + d_N(y,z) \text{ for } (\mu \times \eta)\text{-a.e.} (x,z).
\end{equation} 
For $\varepsilon >0$, we can thus find $x_1' \in B_N(x_1, \varepsilon)$, $x_2' \in B_N(x_2, \varepsilon)$ and $z \in \supp (\eta)$ such that 
\begin{equation}
\label{eq:halflines}
d_N(x_i',z)=d_N(x_i',y)+d_N(y,z)
\end{equation}
holds for $i = 1, 2$.
It is known (see eg. \cite[Lemma 7.2.1]{Pap}) that if $(X,d_N)$ is a strictly convex normed vector space, then three points $x,y,z$ satisfy $$N(x-y)+ N(y-z) = N(x-z)$$ if and only if there exists $t$ in $[0,1]$ such that $y=(1-t)x+tz$. 
\\
Using this, we have $d_N(x_1',z)=d_N(x_1',y) + d_N(y,z)$ if and only if $z$ is on the half-line starting at $y$ given by $L_1(t)=y+t(y-x_1')$, $t\geq 0$. 
 Similarly $d_N(x_2',z)=d_N(x_2',y) + d_N(y,z)$ if and only if $z$ is on the half-line $L_2(t)=y+t(y-x_2')$, $t\geq 0$. If $\varepsilon$ is small enough, these half-lines are distinct and the two equalities of \eqref{eq:halflines} hold only when $z=y$. Thus,  $d_N(x,z)=d_N(x,y)+d_N(y,z)$ holds for $\eta$-a.e. $z$ only if $\eta=\delta_y$. But then $d_{\mathcal{W}_1}(\mu, \delta_y)=0$, and $x_1=x_2=y$, thus $\mu=\delta_y$ is a Dirac mass, which gives the desired contradiction.
\end{proof}
As before, we have the following corollary:
\begin{corollary}\label{corr:p=1 Dirac fix}
Let $p=1$, and consider the metric space $(\mathbb{R}^n,d_N)$, with $d_N$ a distance induced by a strictly convex norm $N$. Assume that $\Phi:\mathcal{W}_1(\Rn,d_N)\to\mathcal{W}_1(\Rn,d_N)$ is an isometry. Then there exists an isometry $\psi:(\Rn,d_N)\to(\Rn,d_N)$ such that $(\Phi\circ\psi_{\#}^{-1})(\delta_x)=\delta_x$ for all $x\in\Rn$.
\end{corollary}

%% file: Rigid_upgrade_3.tex
\section{Upgrading rigidity}
\label{sec:upgrading}

As mentioned in the introduction, in \cite{GTV1} it is shown that, for $p\neq 2$, the Wasserstein space $\mathcal{W}_p(\R, |\cdot |)$ is rigid. As a consequence, if an isometry $\Phi$ of $\mathcal{W}_p(\Rn, d_N)$ globally preserves measures that are supported on a line, then it acts as a trivial isometry on those measures. 
One of the main points of our proof will be to show that such a line exists; this will be the result of the next section. 

In this section, we show how we can utilize rigidity on a proper linear subspace to prove rigidity on the entire space. In particular, we show that, if an isometry $\Phi$ fixes both measures supported on such a subspace and measures supported on an appropriately chosen "complementary" linear subspace, then $\Phi$ is the identity on the whole Wasserstein space. 

We recall that a linear subspace is a space $L \subset \Rn$ such that, if $v_1, v_2 \in L$ and $\lambda_1, \lambda_2 \in \R$, then $\lambda_1 v_1 + \lambda_2 v_2 \in L$. If $L$ is a linear subspace and $v_0 \in \Rn$ is a vector, then the space $v_0 + L$ is called an affine subspace.

Consider a proper linear subspace $L \subset \Rn$. 
We say that a norm $N$ \textit{projects uniquely onto} the linear subspace $L$ if, for any $x \in \Rn$, there exists a unique $\hat{x} \in L$ such that $d_N(x, \hat{x}) \leq d_N(x, y)$ for all $y \in L$. Then we denote $P_{L}(x)=\hat{x}$ the projection map of $x \in \R^n$ onto $L$.  Since for any $v_0 \in \Rn$, we have $d_N(x, y-v_0 )= d_N(x + v_0 , y)$, it is clear that if $N$ projects uniquely onto a linear subspace $L$, then it also projects uniquely onto the affine subspace $v_0 + L$, and we define $P_{v_0 + L}(x)$ analogously to the linear subspace case.

In this section we will assume that the norm $N$ projects uniquely onto the proper subspace $L$. It is easy to see that, if the norm is strictly convex as in the Theorems \ref{thm1} and \ref{thm2}, this assumption always holds.

Notice that the result presented in this section holds for all $1 \leq p < \infty$, including the special case $p=2$.  

We present a few properties of the projection operators. The following lemma is due to Fletcher and Moors \cite{fletcher_chebyshev_2015}:
\begin{lemma}
\label{lem: Chebychev}
If $L$ is a linear subspace of $(\Rn, d_N)$, then for any $x \in \Rn$, $k \in L$ and $\lambda \in \R$, we have 
$$
P_L(\lambda x + k)=\lambda P_L(x)+k.
$$
\end{lemma}
In particular, this Lemma implies that the set $S=P_L^{-1}(0)$ of points that project onto $0$ is homogeneous; if $x \in S$, then the entire line supporting the segment between $0$ and $x$ is contained in $S$. Also, for any $k \in L$, the preimage of $k$, $S_k=P_L^{-1}(k)$ is simply the preimage of $0$ translated by the vector $k$, i.e. $S_k=S + k$. 

If we know the projection of a vector onto a subspace, we can say the following about its projection onto the translation of the subspace:

\begin{lemma} 
\label{lem: proj_trans}
If $L$ is a linear subspace and $v_0, v_1  \in \Rn$, $L'=v_0 + L$, we have $$P_{L'-v_1}(x)=P_{L'}(x+v_1)-v_1$$ for any $x \in \Rn$.
\end{lemma}
\begin{proof}
Set $\hat{x}=P_{L'}(x)$. Then, for any $y \in L'$, $d_N(x, y)=d_N(x-v_1, y-v_1)$, and the unique minimum of the right expression over all $y - v_1\in L'-v_1$ is attained at $y=\hat{x}$. Thus we have $P_{L'-v_1}(x-v_1)=\hat{x}-v_1=P_{L'}(x)-v_1$. 
\end{proof}

We can now rewrite Lemma \ref{lem: Chebychev} in the case of an affine subspace.
\begin{lemma}
\label{lem: Affine_Chebychev}
If $L$ is a linear subspace of $(\Rn, d_N)$ and $v_0 \in \Rn$, then for any $x \in \Rn$, $\lambda \in \R$ and $k \in L$, if $P_{L + v_0}(x)=\hat{x}$, we have 
$$
P_{L+v_0}(\hat{x} + \lambda (x-\hat{x}) + k)= \hat{x}+k.
$$
\end{lemma}
\begin{proof}
Since $\hat{x} \in L + v_0$, we have that $v_0 +  L-\hat{x}=L$, and using Lemmata \ref{lem: Chebychev} and \ref{lem: proj_trans} we get
\begin{align*}
P_{L + v_0}(\hat{x} + \lambda (x-\hat{x})+k)=P_{L + v_0-\hat{x}}(\lambda (x-\hat{x})+k)  + \hat{x} \\
=\lambda P_{L}(x-\hat{x}) + \hat{x}+k.
=\lambda (P_{L+\hat{x}}(x)-\hat{x}) + \hat{x}+k=\hat{x}+k.
\end{align*}
\end{proof}
Thus, like in the linear case, the preimage of a point $S_k=P_{L'}^{-1}(k)$ is a collection of lines passing through $k$, and is a translation of the preimage of $v_0$, $P_{L'}^{-1}(v_0)$.

We write the set of measures supported on the affine subspace $L'$ as $\mathcal{W}_p(L', d_{N, L'})$.
The next lemma shows that the push-forward of the projection operator onto $L'$ defines a projection operator from $\mathcal{W}_p(\Rn, d_N)$ to $\mathcal{W}_p(L', d_{N, L'})$. 

\begin{lemma}
\label{lem: proj meas minim}
Consider the affine subspace $L' \subset \Rn$ and let $\mu \in \mathcal{W}_p(\R^n, d_N)$. Then $\hat{\mu}=P_{L'\#}\mu$ is the unique measure in $\mathcal{W}_p(L', d_{N,L'})$ such that 
$$ d_{W_p}(\mu, \hat{\mu}) \leq d_{W_p}(\mu, \nu)$$ for all $\nu \in \mathcal{W}_p(L', d_{N,L'})$. 
\end{lemma}
\begin{proof}
We recall that a set $\Gamma \in \Rn \times \Rn$ is called $c$-cyclically monotone for a cost function $c:\Rn \times \Rn \to \R$ if for any set of points $(x_i, y_i)_{i=1}^M \subset  \Gamma$ with $M \geq 1$, we have 
$$\sum_{i=1}^Mc(x_i, y_{i+1})\geq \sum_{i=1}^Mc(x_i, y_i),$$ where we define $x_{M+1}=x_1$. 
We first look at the set $$S=\lbrace (x, P_{L'}(x)):x \in \R^n \rbrace$$ and show that it is $c$-cyclically monotone with respect to the cost $c(x,y)=d_N^p(x,y)$.

To do that, take a set of points $(x_i, P_{L'}(x_i))_{i=1}^M \subset S$. Then, by the property of the projection operator, we have that 
$$d_N^p(x_i, P_{L'}(x_i)) \leq d_N^p(x_{i}, P_{L'}(x_{i+1})) $$ for any $1 \leq i \leq M$ (where we set $x_{M+1}=x_1$). By summing over $i$, we get that 
$$
\sum_{i=1}^Md_N^p(x_i, P_{L'}(x_i))  \leq \sum_{i=1}^Md_N^p(x_{i}, P_{L'}(x_{i+1})),   
$$  
showing that $S$ is $c$-cyclically monotone. Applying \cite[Corollary 2.6.8]{Figalli} shows that $P_{L'}$ is an optimal transport map between $\mu$ and $P_{L'\#}\mu=\hat{\mu}$, and therefore we have that 
$$d_{\mathcal{W}_p}^p(\mu, \hat{\mu})= \int_{\R^n}d_N^p(x, P_{L'}(x))d\mu(x).$$ Now, consider a measure $\nu \in \mathcal{W}_p(L', d_{N, L'})$ and $ \pi$ an optimal coupling between $\mu$ and $\nu$. Since by the definition of the projection operator, we have that  
$$d_N(x, P_{L'}(x)) \leq d_N(x,y)$$ for any $y \in L'$, we get that 
\begin{equation}
\label{eq:projmeasminim}
d_{\mathcal{W}_p}^p(\mu, \nu)=\int_{\Rn \times \Rn} d_N^p(x, y) d\pi(x, y) \geq \int_{\Rn \times \Rn} d_N^p(x, P_{L'}(x)) d\pi(x, y)=d_{\mathcal{W}_p}^p(\mu, \hat{\mu}),
\end{equation} and we have the inequality from the lemma. 

To show the uniqueness of $\hat{\mu}$, we notice that the inequality \eqref{eq:projmeasminim} becomes an equality only if $d_N^p(x, y)=d_N^p(x, P_{L'}(x))$ for $\pi$-almost any pair $(x, y)\in \Rn \times L'$. Since the projection onto $L'$ is unique, this implies that $y=P_{L'}(x)$ for $\pi$-almost any pair $(x, y)\in \Rn \times L'$; thus $\pi=(\text{Id} \times P_{L'})_{\#}\mu$, and $\nu=P_{L'\#}\mu=\hat{\mu}$, showing the uniqueness of $\hat{\mu}$. 
\end{proof}

The next lemma shows that, if a Wasserstein isometry $\Phi$ leaves all measures supported on an affine subspace $H'$ invariant, then the projection operator $P_{H'}$ and the isometry $\Phi$ commute. 

\begin{lemma}
\label{lem: projection commutates}
For $p\geq 1$, $L$ a linear subspace, $v_0 \in \Rn$ and $L'=v_0 + L$, let $$\Phi :\mathcal{W}_p(\R^n, d_N) \rightarrow \mathcal{W}_p(\R^n, d_N)$$ be an isometry such that $\Phi (\nu)=\nu$ for all  $\nu \in \mathcal{W}_p(L', d_{N, L'})$. Then we have the relation 
$$
\Phi(P_{L'\#}\mu )=P_{L'\#}\Phi(\mu)
$$
for all $\mu \in \mathcal{W}_p(\R^n, d_N)$. 
\end{lemma} 
\begin{proof}
Let $\mu \in \mathcal{W}_p(\R^n, d_N)$, and set $\hat{\mu}=P_{L'\#}\mu$. By the assumptions, we have $$\Phi (\hat{\mu})=\hat{\mu} \in \mathcal{W}_p(L', d_{N,L'}).$$ Then $d_{\mathcal{W}_p}(\mu, \hat{\mu})=d_{\mathcal{W}_p}(\Phi (\mu), \Phi(\hat{\mu}))=d_{\mathcal{W}_p}(\Phi(\mu), \hat{\mu})$. Using Lemma \ref{lem: proj meas minim}, we have that for $\nu \in \mathcal{W}_p(L', d_{N,L'})$, 
$$d_{\mathcal{W}_p}(\Phi(\mu), \nu) = d_{\mathcal{W}_p}(\mu, \Phi^{-1}(\nu))=d_{\mathcal{W}_p}(\mu, \nu) \geq d_{\mathcal{W}_p}(\mu, \hat{\mu})= d_{\mathcal{W}_p}(\Phi(\mu),\Phi( \hat{\mu})).$$
 Thus $\Phi(\hat{\mu})$ minimizes the expression $d_{\mathcal{W}_p}(\Phi(\mu), \nu)$ among measures $\nu \in \mathcal{W}_p(L' , d_{N, L'})$. By Lemma \ref{lem: proj meas minim}, it is the unique minimizer, and we have that $\Phi (\hat{\mu})=P_{L'\#}\Phi(\mu)$, proving the lemma. 
\end{proof}

This lemma has a very useful consequence. Indeed, consider a linear subspace $L$ such that $\Phi(\mu)=\mu$ for any measure $\mu$ supported on $L$. Then we have $$P_{L \#}(\Phi(\mu))=\Phi(P_{L\#}\mu)=P_{L\#}\mu.$$ Thus, since $\mu$ and $\Phi(\mu)$ have the same projection onto $L$, we have that $\Phi(\mu)$ is supported on the set $ P_L^{-1}(\{0\})+ \text{supp}(\mu)$. 
In Proposition \ref{prop: Hyper implies hyper plus}, the main result of this section, we want to  use this property to conclude that, if the isometry $\Phi$ acts as the identity on measures supported on $L$ and on $ P_L^{-1}(\{0\})$, then $\Phi$ has to be the identity on the whole Wasserstein space over $\Rn$. This approach works very well if the set $ P_L^{-1}(\{0\})$ is a linear subspace; this happens for example if the norm $N$ is an $l_q$-norm, and we choose $L=\{t e_i | t \in \R\}$, if $e_i$ is a canonical base vector. It is then easy to see that $P_L^{-1}(\{0\})$ is the hyperplane $H_i$ defined by $x_i=0$. Unfortunately, as the following example shows, the preimage of the projection operator is not always a linear subspace. 

\begin{example}
\label{ex:l4 case}
We consider the $l_4$ norm on $\R^3$ given by $N_4(x, y, z)=\sqrt[4]{x^4 + y^4 +z^4}$, and the linear subspace $L=\{(t, t, t)| t \in \R\}$. Then $S:=P_L^{-1}(\{0\})$ is not a linear subspace. 
\end{example}
Indeed, we have that 
$$(x, y, z) \in S \Leftrightarrow x^4 + y^4 + z^4 \leq (x-t)^4 + (y-t)^4 + (z-t)^4 \text{ for all } t \in \R.$$

 Using the binomial formula, we can rewrite this inequality as 
\begin{equation}
\label{eq:in_S}
-4(x^3 + y^3 + z^3)t + 6(x^2 + y^2 + z^2)t^2 -4 (x + y + z)t^3 + 3t^4 \geq 0 \text{ for all } t \in \R.
\end{equation}
If $x^3 + y^3 + z^3 \neq 0$, we can take $t$ small enough to guarantee that this inequality does not hold; thus $x^3 + y^3 + z^3=0$ is a necessary condition to have that $(x, y, z) \in S$. 
On the other hand, if $x^3 + y^3 + z^3=0$, then \eqref{eq:in_S} is equivalent to 
\begin{equation}
\label{eq:in_S2}
3(x^2 + y^2 + z^2) -2 (x + y + z)t + \frac{3}{2}t^2 \geq 0 \text{ for all } t \in \R,
\end{equation} which is a quadratic inequality in $t$. Since the property $2xy \leq x^2 + y^2$ implies that $$(x + y + z)^2 \leq 3 (x^2 + y^2 + z^2),$$ we can estimate the discriminant of this quadratic inequality as 
$$
4(x + y + z)^2 -18(x^2+y^2+z^2) \leq -6(x^2+y^2+z^2) \leq 0,
$$
and \eqref{eq:in_S2} is always satisfied. Therefore $x^3 + y^3 + z^3=0$ is also a sufficient condition for $(x,y,z) \in S$, and $S$ is the set of points in $\R^3$ satisfying the equation $x^3 + y^3 + z^3 = 0$. This surface is obviously not a linear subspace; for example, while both $P_1=(1, -1, 0)$ and $P_2=(0, -1, 1)$ are in $S$, $P_1 + P_2$ is not in $S$. 
\begin{figure}
\centering
\includegraphics[scale=0.2]{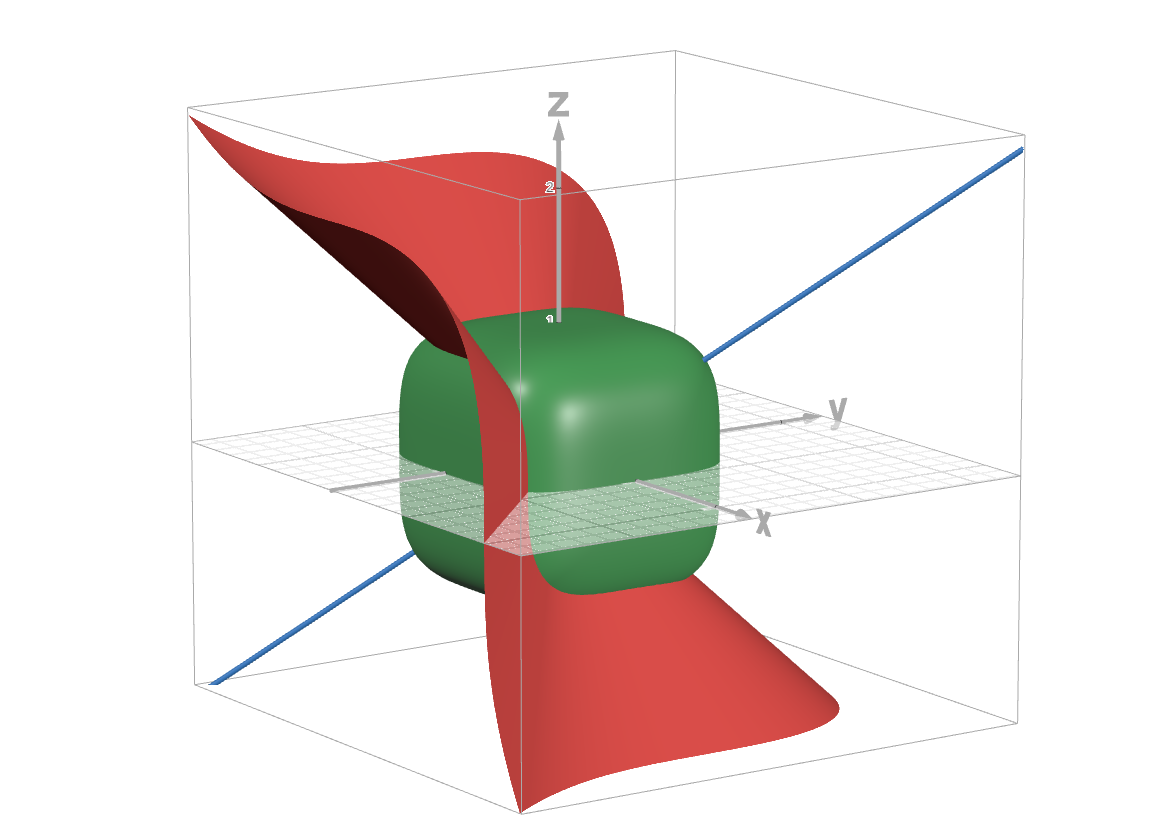}
\caption{In blue the subspace $L$, in green the unit ball of the $l_4$-norm, in red the surface $S=P_L^{-1}(\{0\})$. Created using Desmos.}
\label{fig:surfaceS}
\end{figure}

If $P_L^{-1}(\{0\})$ is not a linear subspace, the idea of Proposition \ref{prop: Hyper implies hyper plus} does not work anymore, since we need the intersection between $P_L^{-1}(\{0\})$ and $L$ to be a single point to show the finite support of $\Phi(\mu)$. Further, the proof of Theorem \ref{thm1} (which will be presented at the end of section \ref{sec:potential}) uses an induction argument on the dimension of $n$. If $P_L^{-1}(\{0\})$ is not a linear subspace, we cannot use the induction assumption to assume that $\Phi$ acts as the identity on $P_L^{-1}(\{0\})$. 

To counter this problem, we show in the next lemma that the geometry of the preimage of the projection operator restricts the possible support points of the image of a measure under an isometry. Indeed, we show that while $P_L^{-1}(\{0\})$ might not be a linear subspace, there exists a linear subspace $H \subset P_L^{-1}(\{0\})$ (which might be trivial, $H=\{0\}$) such that for any measure $\mu \in \mathcal{W}_p(\Rn, d_N)$ and any isometry $\Phi$ of the Wasserstein space fixing pointwise all measures supported on a translate of $L$, the image $\Phi(\mu)$ will be supported on the set $H + \text{supp}(\mu)$.


\begin{lemma}
\label{lem: restrict_image}
Consider a normed space $(\Rn, d_N)$ and a linear subspace $L$ such that $N$ projects uniquely on $L$. Then there exists a linear subspace $H$ such that $P_L(H)= \{0\}$ and that has the following additional property: if  $\Phi$ is an isometry of the Wasserstein space such that for $v_0 \in \Rn$, measures $\nu$ supported on $L + v_0$ are fixed i.e. $\Phi(\nu)=\nu$, then $\text{supp}(\Phi(\mu)) \subset H + \text{supp}(\mu)$ for all $\mu \in \mathcal{W}_p(\Rn, d_N)$.
\end{lemma}

\begin{proof}
Before we start the proof, we can notice that for $n=2$ this lemma greatly simplifies. Indeed, in $\R^2$, the only possible linear subspaces are $\{0\}$, $\R^2$ and lines passing through $0$. In the first two cases, we set $H=\R^2$ or $H=\{0\}$, and the Lemma is trivially true. If the linear subspace $L$ is a line, then by Lemma \ref{lem: Chebychev} the set $P_L^{-1}(0)$ is also a line passing through $0$. Setting $H:=P_L^{-1}(0)$, we immediately have that $H$ is a linear subspace and that $P_L(H)=\{0\}$. The condition  $\text{supp}(\Phi(\mu)) \subset H + \text{supp}(\mu)$ is a simple application of Lemma \ref{lem: projection commutates}. Thus we will now assume that $n\geq 3$. 

We also assume in this proof that any isometry $\Phi$ of the Wasserstein space $\mathcal{W}_p(\Rn, d_N)$ that we consider fixes measures supported on a translated subspace $L$, i.e. $\Phi$ fulfills the assumption of the Lemma. 

We consider the set 
\begin{align*}
\mathcal{G}=\{ \sum_{k=1}^M a_k \delta_{x_k} : M \geq 1, \sum_{k=1}^M a_k=1, x_k \in \Rn \text{ for all } 1 \leq k \leq M \\
\text{ and if } k \neq k',\text{ then } P_{L + x_i}(x_k) \neq P_{L+x_i}(x_{k'}) \text{ for all } i \leq M \}.
\end{align*} 
It is easy to see that this set is dense in $\mathcal{W}_p(\Rn, d_N)$. 

Take $\mu=\sum_{k=1}^M a_k \delta_{x_k} \in \mathcal{G}$.
By Lemma \ref{lem: projection commutates}, for any $v\in \Rn$,   
$$
 P_{L+v\#}(\mu)=\Phi (P_{L+v\#}(\mu))=P_{L+v\#}(\Phi (\mu)).
$$ 
We can then write 
\begin{equation}
\label{eq: proj_v_phi_mu}
P_{L+v\#}(\Phi (\mu))=P_{L+v\#}(\mu)=\sum_{k=1}^Ma_k \delta_{P_{L+v}(x_k)}.
\end{equation}
 Thus, for every $v \in \Rn$ and every point $y \in \text{supp}(\Phi(\mu))$, there exists $k \leq M$ such that $P_{L+v}(y)=P_{L+v}(x_k)$.

Before we continue, we give a brief overview of the steps of the proof. As a first step, we will, for every $i \leq M$, look at the space $S_i=P_{L+x_i}^{-1}(x_i)$, i.e. the set of points such that their projection onto the affine subspace $L+x_i$ is exactly $x_i$. Then, for any point $y \in \text{supp}(\Phi(\mu))$ such that $y \in S_i$, we show that locally, $S_i$ looks the same around $x_i$ as it does around $y$. This means that, for any $v\in \Rn$ small enough, $x_i + v \in S_i$ if and only if $y + v \in S_i$. This is exactly the statement of equation \eqref{eq:proj_equiv}. 

Since the set $S_i$ is homogeneous with center $x_i$, we can extend this local behaviour into a global statement. Thus, as our second step, we show equation \eqref{eq:plane_projects}, which tells us that (unless $y=x_i$) the projection set $S_i$ does not only contain lines but also planes spanned by $y-x_i$ and by any $v \in S_i$. This result in particular restricts the possible support points of $\Phi(\mu)$, since support points can only exists if the set $S_i$ (which up to translation is equal to the set $P_L^{-1}(0)$, and thus only depends on the metric and on the linear subspace $L$) contains these planes; in particular, if $S_i$ does not contain any planes, then necessarily $y=x_i$.

 In the third step, we show that any $y \in \text{supp}(\Phi(\mu))$ can only be associated to one support point of $\mu$, i.e. there exists for any $y$ a unique $i \leq M$ such that $P_{L+v}(y)=P_{L+v}(x_i)$, which is independent of $v$. 
 
To finish the proof we construct $H$ as the span of all possible vectors $y-x_i$ such that for some measure $\mu \in \mathcal{G}$ and some isometry $\Phi$, $x_i \in \text{supp}(\mu)$ and $y \in \text{supp}(\Phi(\mu))$, with $y \in S_i$. By the properties found in the first three steps, the linear subspace $H$ will have the required properties of the Lemma. In particular, if $P_L^{-1}(0)$ does not contain any planes, then for all $y \in \text{supp}(\Phi(\mu))$ there exists $x_i \in \text{supp}(\mu)$ such that $y=x_i$, and thus $H=\{0\}$ is trivial.   

We now continue with the proof. Fix $i \leq M$ and consider the affine subspace $L'=L+x_i$. Then we have trivially that $P_{L'}(x_i)=x_i$ and by the assumptions on $\mathcal{G}$ that $P_{L'}(x_j) \neq x_i $ for all $j \leq M, j\neq i$.

Consider a point $y \in \text{supp}(\Phi(\mu))$ such that $P_{L'}(y)=x_i$. By the continuity of the projection operator (Corr 2.20 of \cite{fletcher_chebyshev_2015}) and equation \eqref{eq: proj_v_phi_mu}, we can show that, if $v\in \Rn$ is small enough, then 
\begin{align}
\label{eq:proj_equiv}
P_{L'}(y+v)=P_{L'}(y) \Leftrightarrow  P_{L'}(x_i+v)=P_{L'}(x_i)=x_i.
\end{align}
For $y=x_i$, this result is trivial. If $y \neq x_i$, this result follows from the continuity of the projection operator and from the previous lemmata. 

Indeed, for any $\varepsilon >0$, there exists $\delta_0 >0$ such that, if $N(v)<\delta_0$, then $$N(P_{L'}(y +v)-P_{L'}(y))\leq \varepsilon /4.$$ By lemma \ref{lem: proj_trans}, we have for every $v $ such that $N(v) < \delta :=\min (\delta_0, \varepsilon /4)$ that $$N(P_{L'-v}(y)-P_{L'}(y)) \leq \varepsilon /2.$$ Choosing $\varepsilon$ such that $d_N(P_{L'}(x_i),P_{L'}(x_j))> \varepsilon$ for any $j \neq i$, we have that for $j \neq i$,
\begin{align*}
N[P_{L'-v}(y)-P_{L'-v}(x_j)] \geq N[P_{L'-v}(x_i)-P_{L'-v}(x_j)]-   N[P_{L'-v}(y)-P_{L'-v}(x_i)] \\
 \geq N[P_{L'}(x_i ) -P_{L'}(x_j)]- N[P_{L'}(x_i + v)-P_{L'}(x_i)]-N[P_{L'}(x_j+v)-P_{L'}(x_j)] \\
-   N[P_{L'}(y+v)- P_{L'}(y)]- N[P_{L'}(x_i+v)-P_{L'}(x_i)] \\
 >\varepsilon - \varepsilon/4 - \varepsilon/4 -\varepsilon/4 -\varepsilon/4=0.
\end{align*}
Thus $P_{L'-v}(y) \neq P_{L'-v}(x_j)$ when $j \neq i$.
But equation \eqref{eq: proj_v_phi_mu} tells us that $P_{L'-v}(y)=P_{L'-v}(x_j)$ for some $j\leq M$. Thus we have that, for any $v$ small enough, if $y\in \supp(\Phi(\mu))$ and if $P_{L'}(y)=P_{L'}(x_i)$, then $P_{L'-v}(y)=P_{L'-v}(x_i)$. As a consequence of this, assume that $v$ is small enough and that $P_{L'}(x_i+v)=P_{L'}(x_i)$. Then we have 
\begin{align*}
P_{L'}(y+v)=P_{L'-v}(y) +v = P_{L'-v}(x_i)+v=P_{L'}(x_i)=P_{L'}(y). 
\end{align*} The reverse is also true. Indeed, if $P_{L'}(y+v)=P_{L'}(y)$, then 
\begin{align*}
P_{L'}(x_i+v)=P_{L'-v}(x_i) -v = P_{L'-v}(y)-v=P_{L'}(y)=P_{L'}(x_i),
\end{align*}
showing equation \eqref{eq:proj_equiv}.

With equation \eqref{eq:proj_equiv} we can now show that, if $v \in \Rn \setminus \{0\}$ is such that $P_{L'}(x_i + v)=x_i$, then $P_{L'}(x_i + z)=x_i$ holds for any $z$ in the set generated by $y-x_i$ and $v$ (which is a plane, unless $y=x_i$ or $v$ is colinear to $x_i-y$). 

We first notice that, if $P_{L'}(x_i + v)=x_i$, then by Lemma \ref{lem: Affine_Chebychev} we get $P_{L'}(x_i + \lambda' v)=x_i$ for any $\lambda' \in \R$; thus we can assume w.l.o.g. that $v$ is small enough. Then we have by equation \eqref{eq:proj_equiv} that $P_{L'}(y + v)=x_i$. Again by Lemma \ref{lem: Affine_Chebychev}, $P_{L'}(y +  \lambda v)=x_i$ still holds for any $|\lambda| \leq 1$. A final application of Lemma \ref{lem: Affine_Chebychev} then gives that 
\begin{equation}
\label{eq:alpha_projection}
P_{L'}(x_i + \alpha (y-x_i +  \lambda v))=x_i
\end{equation} for any $\alpha \in \R$. 

Thus, if $z=\lambda_1 (y-x_i) + \lambda_2 v \neq 0$ and $\delta_1 = \min \{\frac{\delta}{2N(z)}, \frac{1}{2|\lambda_1| + 2|\lambda_2|} \}$, set $z'=\delta_1 z$.
Notice that by Lemma \ref{lem: Affine_Chebychev} and equation \eqref{eq:proj_equiv},
$$
P_{L'}(x_i + z)=x_i \Leftrightarrow P_{L'}(x_i +z')=x_i \Leftrightarrow P_{L'}(y + z')=x_i.
$$  
Since 
$$ y + z'= y + \delta_1(\lambda_1 (y-x_i) + \lambda_2 v)=x_i + (1 +  \delta_1 \lambda_1)(y-x_i) +  \delta_1 \lambda_2 v,$$  
we can write $y+z'=x_i +\alpha (y-x_i) + \lambda \alpha v$ (where the definition of $\delta_1$ guarantees that $|\lambda|<1$). Using equation \eqref{eq:alpha_projection}, this shows that $P_{L'}(y + z')=x_i$ and
\begin{equation}
\label{eq:plane_projects}
P_{L'}(x_i + z )=x_i.
\end{equation}

This also implies that if, for $y \in \text{supp}(\Phi(\mu))$, $P_{L+x_i}(y)=x_i$ and $P_{L+x_j}(y)=x_j$, then $i=j$. Indeed, if $P_{L+x_i}(y)=x_i$, since for $x_i^j:=P_{L+x_j}(x_i)$,  $P_{L+x_i}(x_i^j)=x_i$, by the above result we have that 
$$
P_{L+x_i}(x_i + \lambda_1 y + \lambda_2 x_i^j)=x_i. 
$$ 
But then 
$$
P_{L+x_j}(y)= P_{L+x_j -(x_i^j-x_i)}(y-(x_i^j-x_i)) + (x_i^j-x_i)=P_{L+x_i}(x_1  + y -x_i^j)+ (x_i^j-x_i)=x_i^j,
$$
where we used that $x_j-x_i^j \in L$. 
This shows that, for all $j \leq M$, if $P_{L+x_i}(y)=x_i$, then $P_{L+x_j}(y)=P_{L+x_j}(x_i) \neq x_j$ unless $i=j$. 
Thus, if we set $Y_{x_i}=P^{-1}_{L+x_i}(x_i) \cap \text{supp}(\Phi(\mu))$, then for any $i \neq j$, the intersection $Y_{x_i} \cap Y_{x_j}$ will be empty. 
From equation \eqref{eq: proj_v_phi_mu} we also have that $\Phi(\mu)(Y_{x_i})=\mu(x_i)$. Since
$$
1=\sum_{i=1}^M \mu(x_i)=\sum_{i=1}^M \Phi(\mu)(Y_{x_i}),
$$
we have that $\sqcup_{i=1}^M Y_{x_i}= \text{supp}(\Phi(\mu))$. Thus, for all $y \in \text{supp}(\Phi(\mu))$, there exists $i \leq M $ with $P_{L+x_i}(y)=x_i$. 

We are now ready to finish the argument.  
Since $P_{L + x_i}(z)=x_i \Leftrightarrow P_{L}(z-x_i)=0$, for any isometry of the Wasserstein space $\Phi$ and any measure $\mu \in \mathcal{G}$, if $y \in \text{supp}(\Phi(\mu))$, then there exists $x_i \in \text{supp}(\mu)$ such that $y \in Y_{x_i}$ and 
\begin{equation}
\label{eq:onequation}
P_{L}(y-x_i)=0.
\end{equation} 
 Then, for any $v \in \Rn$ such that $P_L(v)=0$, we have by equation \eqref{eq:plane_projects} that for any $\lambda_1, \lambda_2 \in \R$,  
\begin{equation}
\label{eq:twoequation}
P_{L}(\lambda_1 (y-x_1)+ \lambda_2v)=0.
\end{equation} We consider the set 
\begin{align*}
 V=\{v \in \Rn | \exists \Phi \in \text{Isom}(\mathcal{W}_p(\Rn, d_N)), \, \exists \mu \in \mathcal{G},  \\
 \exists y \in \text{supp}(\Phi(\mu)), \exists x_i \in \text{supp}(\mu) \text{ such that } y \in Y_{x_i}, \, v=y-x_i \}. 
 \end{align*}
 In words, $V$ is the set of vectors $v$ such that there exists an isometry $\Phi$ of the Wasserstein space (which fixes measures supported on translated subspaces of $L$) and a measure $\mu \in \mathcal{G}$ such that, for some $y$ in the support of $\Phi(\mu)$, if $P_{L+x_i}(y)=x_i$  for some $x_i \in \text{supp}(\mu)$, then $v=y-x_i$. 
 
 We now consider the set $H = \text{span}(V)$. It is clear that $H$ is a linear subspace. To see that $P_L(H)=0$, we first notice that if $v \in V$, there exists an isometry $\Phi$, $\mu\in \mathcal{G}$, $x_i \in \text{supp}(\mu)$ and $y \in \text{supp}(\Phi(\mu))$ such that $y \in Y_{x_i}$ $v=y-x_i$ and \eqref{eq:onequation} implies that $P_L(v)=0$.  Then, for $v_1, v_2 \in V$, setting  $z=\lambda_1 v_1 + \lambda_2 v_2$, equation \eqref{eq:twoequation} implies that $P_L(z)=0$. Given a basis of $H$ composed of $v_{y_1}, \dots, v_{y_m}$, an induction argument shows that for all $z \in H$ we have $P_{L}(z)=0$. 
 
Using the same method, we have that, for any $h \in H$,
 \begin{equation}
 \label{eq:same_plane}
 P_L(h + x_i)=P_L(h + (x_i -P_L(x_i)) + P_L(x_i))=P_L(h + (x_i-P_L(x_i))) + P_L(x_i)=P_L(x_i)
 \end{equation}
 where we used Lemma \ref{lem: Chebychev} and that $P_L(x_i -P_L(x_i))=P_L(x_i) -P_L(x_i)=0$. 

Thus we have a linear subspace $H$ such that $P_{L}(H)=\{0\}$ and, for any isometry $\Phi$ and $\mu \in \mathcal{G}$, if $y \in \text{supp}(\Phi(\mu))$, then there exists $x_i \in \text{supp}(\mu)$ such that $y-x_i \in H$. A density argument shows the result for all $\mu \in \mathcal{W}_p(\Rn, d_N)$, proving the Lemma.
\end{proof}

To see an application of this Lemma, we can look again at Example \ref{ex:l4 case} and try to explicitly determine $H$. Consider an isometry $\Phi$, a measure $\mu \in \mathcal{W}_p(\Rn, d_4)$, $x_i \in \text{supp}(\mu)$ and a point $y\in \text{supp}(\Phi(\mu))$ such that $P_L(y-x_i)=0$. Then, for any $v$ such that $P_L(v)=0$, if we consider the linear subspace $K=\text{span}((y-x_i), v)$, by equation \eqref{eq:twoequation} we have $P_L(K)=0$, and the preimage $P_L^{-1}(\{0\})$ must contain the linear subspace $K$. If $x_i\neq y$, we can choose $v$ such that $x_i-y$ and $v$ are not collinear, and $K$ is a plane. Since in the case presented in Example \ref{ex:l4 case} the preimage $P_L^{-1}(\{0\})$ does not contain any plane, this means that $y=x_i$; in other words, we have $V=\{0\}=H$, and Lemma \ref{lem: restrict_image} implies that if an isometry $\Phi$ of the Wasserstein space acts as the identity on measures supported on $L+v_0$, then for any measure $\mu,$ $\text{supp}(\Phi(\mu))\subseteq \text{supp}(\mu)$ holds, i.e. the support of $\Phi(\mu)$ is contained in the support of $\mu$ (and it is easy to see that actually $\mu$ and $\Phi(\mu)$ have the same support). 

We are now ready to prove the main proposition of this section, in which we show that rigidity on certain subspaces of $\Rn$ is enough to prove rigidity on the whole space. This Proposition generalizes Proposition 2.1 from \cite{BKTV}, which showed the result in two dimensions for the special case of the maximum norm. The main difference in our proof comes from the much more general nature of the preimage of a projection operator, which in 2 dimensions is simply a line. We recommend that readers who want to understand this proof in details read first the Proposition 2.1 from \cite{BKTV}, which presents the ideas of the proof in a simpler setting.  

\begin{proposition}
\label{prop: Hyper implies hyper plus}
For $p \geq 1$, assume that there exists a linear subspace $L \subset \Rn$ such that $N$ projects uniquely on $L$ and let $\Phi: W_p(\Rn,d_N) \to W_p(\Rn,d_N)$ be an isometry such that $\Phi(\mu)=\mu$ for every measure $\mu$ supported on a translate of L. 
 Assume also that for the linear subspace $H \subset P_{L}^{-1}(0)$ given by Lemma \ref{lem: restrict_image}, $N$ projects uniquely onto $H$ and, for any $\nu \in \mathcal{W}_p(H, d_{N, H})$, we have $\Phi(\nu)=\nu$. Then $\Phi (\mu)=\mu$ for all $\mu \in \mathcal{W}_p(\Rn, d_{N})$.
\end{proposition}

\begin{proof}

We start with the case $H=0$. We build the subset $\mathcal{F}_0 \subset \mathcal{W}_p(\Rn, d_N)$ as follows: 
\begin{align*}
\mathcal{F}_0=\{ \sum_{k=1}^M a_k \delta_{x_k} : M \geq 1, \sum_{k=1}^M a_k=1, x_k \in \Rn \text{ for all } 1 \leq k \leq M \\
\text{ and if } k \neq k',\text{ then } a_k \neq a_{k'} \text{ and } P_{L}(x_k) \neq P_{L}(x_{k'})  \}.
\end{align*}
It is well known (see e.g. \cite[Theorem 6.18]{Villani}) that measures in $\mathcal{W}_p(\Rn, d_{N})$ can be approximated by a finite combination of Dirac masses. Since it is easy to see that measures with finite support can be approximated by measures of $\mathcal{F}_0$, we have that $\mathcal{F}_0$ is dense in $\mathcal{W}_p(\Rn, d_N)$. 

Consider $\mu\in \mathcal{F}_0$, $\mu=\sum_{k=1}^M a_k \delta_{x_k}$. Since we consider the case $H=\{0\}$, Lemma \ref{lem: restrict_image} says that $\text{supp}(\mu)=\text{supp}(\Phi(\mu))$, and we can write $\Phi(\mu)=\sum_{k=1}^M b_k \delta_{x_k}$. 

We now use Lemma \ref{lem: projection commutates} to show that 
$$\sum_{k=1}^Mb_k \delta_{P_{L}(x_k)}=P_{L\#}(\Phi (\mu))=P_{L\#}(\mu)=\sum_{k=1}^Ma_k \delta_{P_{L}(x_k)}.$$
Since by the condition on $\mathcal{F}_0$, $\delta_{P_{L}(x_k)}  \neq \delta_{P_{L}(x_{k'})}$ when $k \neq k'$, this implies that $b_k=a_k$ for any $1 \leq k \leq M$, and thus $\Phi(\mu)=\mu$. Since $\mathcal{F}_0$ is dense in $\mathcal{W}_P(\Rn, d_N)$, this shows the Proposition when $H=\{0\}$. 

When $H$ is  non-trivial, we want to use the same idea. However, we need to show that even in this case we have $\text{supp}(\mu)=\text{supp}(\Phi(\mu))$; this will be the main part of the rest of the proof. 

We start by building the subset $\mathcal{F} \subset \mathcal{W}_p(\Rn, d_N)$ as follows: 
\begin{align*}
\mathcal{F}=\{ \sum_{k=1}^M a_k \delta_{x_k} : M \geq 1, \sum_{k=1}^M a_k=1, x_k \in \Rn \text{ for all } 1 \leq k \leq M \\
\text{ and if } k \neq k',\text{ then } a_k \neq a_{k'} \text{ and } P_{L}(x_k) \neq P_{L}(x_{k'}), P_{H}(x_k) \neq P_{H}(x_{k'})  \}.
\end{align*}

We show that $\mathcal{F}$ has three properties. First we check that $\mathcal{F}$ is dense in $\mathcal{W}_p(\Rn, d_{N})$. It is well known (see e.g. \cite[Theorem 6.18]{Villani}) that measures in $\mathcal{W}_p(\Rn, d_{N})$ can be approximated by a finite combination of Dirac masses. Since it is clear that any finitely supported measure can be approximated by measures of $\mathcal{F}$, this shows the required density. 

For the second property we check that the maps $P_{L \#}$ and $P_{H \#}$ from $\mathcal{F}$ to the set of measures supported on $L$ and $H$ respectively are injective. 
 We show the injectivity for the map $P_{L \#}$, the second case is similar. Assume that $ \mu_1, \mu_2 \in \mathcal{F}$ with $P_{L\#}\mu_1=P_{L\#}\mu_2$.  We write 
$$
\mu_1= \sum_{k=1}^{M_1}a^1_{k}\delta_{x^1_{k}} \text{ and } \mu_2= \sum_{k=1}^{M_2}a^2_{k}\delta_{x^2_{k}}.
$$
Using that $P_{L\#}\mu_1=P_{L\#}\mu_2$, we get
$$
\sum_{k=1}^{M_1}a^1_{k}\delta_{P_{L}(x^1_{k})}= \sum_{k=1}^{M_2}a^2_{k}\delta_{P_{L}(x^2_{k})}.
$$ 
Due to the conditions of the set $\mathcal{F}$, we have that $M_1=M_2$, $a^1_{k}=a^2_{k}$ and $x^1_{k}=x^2_{k}$ for $1 \leq k \leq M_1$. Therefore $\mu_1=\mu_2$, and $ \mu \rightarrow P_{L\#}\mu$ (and by the same reasoning $ \mu \rightarrow P_{H\#}\mu$) is injective on the set $\mathcal{F}$. 

Finally we want to show that $\Phi (\mathcal{F}) \subset \mathcal{F}$. Take $\mu=\sum_{k=1}^M a_k \delta_{x_k} \in \mathcal{F}$.

By Lemma \ref{lem: projection commutates},   
$$
 P_{L\#}(\mu)=\Phi (P_{L\#}(\mu))=P_{L\#}(\Phi (\mu)) \text{ and } P_{H\#}(\mu)=\Phi (P_{H\#}(\mu))=P_{H\#}(\Phi (\mu)).
$$ 
We can then write 
\begin{equation}
\label{eq: proj_phi_mu}
P_{L\#}(\Phi (\mu))=P_{L\#}(\mu)=\sum_{k=1}^Ma_k \delta_{P_{L}(x_k)}
\end{equation} and similarly
\begin{equation}
\label{eq: proj_phi_mu_also}
P_{H\#}(\Phi (\mu))=P_{H\#}(\mu)=\sum_{k=1}^Ma_k \delta_{P_{H}(x_k)}.
\end{equation}

By the Lemma \ref{lem: restrict_image} and equation \eqref{eq: proj_phi_mu_also}, we have that $\Phi(\mu)$ is supported on the set 
$$
S= \left(\bigcup_{i=1}^M (H + x_i)\right) \bigcap \left( \bigcup_{i=1}^M P_{H}^{-1}(P_H(x_i)) \right). 
$$
Since $H$ is a linear subspace, given a pair $(k, k')$, the intersection $(H+ x_k) \cap P_H^{-1}(P_H(x_{k'}))$ is a unique point, which we denote by $x_{k, k'}.$ By the definition of the set $S$ and equation \eqref{eq:same_plane} we have that, for any pair $(k,k')$, $P_L(x_{k, k'})=P_L(x_k)$ and $P_H(x_{k, k'})=P_H(x_{k'})$. 

Thus $\Phi(\mu)$ is supported on the points of the form $x_{k, k'}$, and we can write
$$
\Phi(\mu)=\sum_{k=1}^M \sum_{k'=1}^M b_{k, k'} \delta_{x_{k,k'}}.
$$ 

Figure \ref{fig:M2} gives an example of a measure $\mu$ and an (a priori) possible image measure\footnote{ We show in this proof that the right measure cannot actually be the image of the left measure under an isometry.}.
\input{mu_and_Phimu_images.tex}
We now give a brief sketch of the rest of the proof. We will assume by contradiction that, as in Figure \ref{fig:M2}, there exist two distinct points $z$ and $z'$ in the support of $\Phi(\mu)$ that project onto $P_L(x_1)$; if there are no such two points, a short argument then shows that $\Phi(\mu)=\mu\in \mathcal{F}$. We then slightly perturb the measure $\mu$ to create $\mu'$, by moving a small weight from $x_1$ to some point $x_0$ close to $x_1$. Similarly, we perturb $\nu=\Phi(\mu)$ twice, by shifting a small weight from either $z$ or $z'$, to a point close to either $z$ or $z'$, creating two new and distinct measures $\nu_1'$ and $\nu_2'$. These new measures are represented in Figure \ref{fig:M1}. 

The constructed measure $\mu'$ has an important property. Indeed, if we call the \textit{fingerprint} $\mathcal{R}$ of a measure $\xi$ its projection onto $L$ and $H$, $\mathcal{R}(\xi)=(P_{L \#}\xi, P_{H \#} \xi)$, then there exists no other measure $\xi$ that $a)$ has the same fingerprint as $\mu'$ and $b)$ lies at the same distance from $\mu$ as $\mu'$. In other words, if 
$$\mathcal{R}(\xi)= \mathcal{\mu'} \text{ and }d_{\mathcal{W}_p}(\mu, \xi)=d_{\mathcal{W}_p}(\mu, \mu'),$$
then we have that $\xi=\mu'$. To get the desired contradiction, we will remark that both $\nu_1'$ and $\nu_2'$ have by their construction the same fingerprint as $\mu'$, and are both at the same distance from $\nu$ as $\mu'$ is from $\mu$, 
$$
d_{\mathcal{W}_p}(\nu, \nu_1')=d_{\mathcal{W}_p}(\nu, \nu_2')=d_{\mathcal{W}_p}(\mu, \mu').
$$ 
Also, by Lemma \ref{lem: projection commutates}, the fingerprint of a measure is preserved under isometries, i.e. $\mathcal{R}(\xi)=\mathcal{R}(\Phi(\xi))$. Thus, the images of $\nu_1'$ and $\nu_2'$ under the inverse isometry $\Phi^{-1}$ have the same fingerprint as $\mu'$, and (since $\Phi^{-1}$ is also an isometry) lie at the same distance from $\Phi^{-1}(\nu)=\mu$ as $\mu'$, or 
\begin{align*}
\mathcal{R}(\Phi^{-1}(\nu_1'))&=\mathcal{R}(\Phi^{-1}(\nu_2'))=\mathcal{R}(\Phi^{-1}(\mu')), \\
 d_{\mathcal{W}_p}(\mu, \Phi^{-1}(\nu_1'))&=d_{\mathcal{W}_p}(\mu, \Phi^{-1}(\nu_2'))=d_{\mathcal{W}_p}(\mu, \mu'). 
\end{align*} 
But then we necessarily have $\Phi^{-1}(\nu_1')=\Phi^{-1}(\nu_2')=\mu'$, and $\nu_1'=\nu_2'$, giving the desired contradiction. 
Having given this brief sketch, we continue with the proof.

Assume for contradiction that $\Phi(\mu) \notin \mathcal{F}$. Then there exists $k_0 \leq M$ and $k_1 \neq k_2 \leq M$ such that both $b_{k_0, k_1}$ and $ b_{k_0, k_2}$ are non-zero.
Indeed, assume that for any $k \leq M$, there exists a unique $k' \leq M$ such that $b_{k, k'} \neq 0$. Since from equation \eqref{eq: proj_phi_mu} we have that 
$$
\sum_{k'=1}^Mb_{k,k'}=a_k,
$$
this shows that $b_{k,k'}= a_k$, and $\Phi(\mu)$ is supported on $M$ points. Again from \eqref{eq: proj_phi_mu} we can then deduce that actually $\Phi(\mu)=\mu \in \mathcal{F}$. 
Thus we have that there exists $k_0 \leq M$ such that, for some $k_1 \neq k_2 \leq M$, we have $b_{k_0, k_1}$ and $b_{k_0, k_2}$ are non-zero. Without loss of generality we can assume that $k_0=1$.  

We set $h$ to be the shortest distance between any two points $x_{k, k'}$, i.e. $$h=\min_{(k,k') \neq (\tilde{k}, \tilde{k'})} d_N(x_{k,k'}, x_{\tilde{k}, \tilde{k'}}).$$
Since the set $\{ x_{k, k'} |k,k' \leq M \}$ is discrete and finite, we have that $h>0$.

\noindent We build the point $x_0 =x_1 + \frac{h_0}{d_N(x_1, P_{L}(x_1))}(x_1-P_{L}(x_1))$ where $h_0\in \R$ is such that $0 < h_0<\frac{h}{2}$. 
We then consider the points $x_{0, k}$ defined such that $P_H(x_{0, k})=P_H(x_0)$ and $P_{L}(x_{0, k})=P_{L}(x_k)$. Explicitly, these points are given by $x_{0, k'}=x_{k'} + P_H(x_0)-P_H(x_{k'})$. In particular $x_{0, 1}=x_0$. Also, we notice that $d_N(x_{0, k'}, x_{1, k'})=h_0$, this follows easily from Lemma \ref{lem: Chebychev}. 

Consider the weight $a_0=\frac{1}{2}\min (b_{1, k_1}, b_{1,k_2})$. By our assumption $b_{1, k_1}, b_{1,k_2}>0$, and we thus have that $a_1 > a_0>0$.  We build the following measures:
\begin{align*}
\mu' &= a_0 \delta_{x_{0}} + (a_1-a_0)\delta_{x_1} + \sum_{k=2}^Ma_k \delta_{x_k}, \\
\nu_1'&=a_0 \delta_{x_{0, k_1}} +(b_{1, k_1}-a_0)\delta_{x_{1, k_1}}+ \sum_{k=1, k\neq k_1}b_{1, k}\delta_{x_{1, k}} + \sum_{k=2}^M\sum_{k'=1}^Mb_{k, k'}\delta_{x_{k,k'}}, \\
\nu_2'&=a_0 \delta_{x_{0, k_2}} + (b_{1, k_2}-a_0)\delta_{x_{1, k_2}} + \sum_{k=1, k\neq k_2}b_{1, k}\delta_{x_{1, k}} + \sum_{k=2}^M\sum_{k'=1}^Mb_{k, k'}\delta_{x_{k,k'}}.
\end{align*}

The measure $\mu'$ (resp. $\nu_1'$ and $\nu_2'$) is obtained by "shifting" a small portion of the weight of $\mu$ (resp. $\Phi(\mu)$) from 
$x_1$ to $x_0$ (resp. from $x_{1, k_{1/2}}$ to $x_{0, k_{1/2}}$). Thus the projections of $\mu'$ onto $L$ and $H$ are 
$$
P_{L\#}(\mu')=a_0 \delta_{P_{L}(x_0)} + (a_1-a_0)\delta_{P_{L}(x_1)} +\sum_{k=2}^Ma_k\delta_{P_{L}(x_k)}
$$
and
$$
P_{H\#}(\mu')=P_{H\#}(\mu)=\sum_{k=1}^M a_k \delta_{P_{H}(x_k)};
$$
the measures $\nu_1', \nu_2'$ have the same respective projections. 
\input{mu_and_nu_prime.tex}

Then we can show that
$$
d_{\mathcal{W}_p}(\mu, \mu')=d_{\mathcal{W}_p}(\Phi(\mu), \nu_1')=d_{\mathcal{W}_p}(\Phi(\mu), \nu_2')=a_0^{\frac{1}{p}}h_0.
$$
For this, consider an optimal transport plan $\pi_0\in \Pi(\mu, \mu')$. Then, as $\mu$ and $\mu'$ are a combination of Dirac masses, we have
\begin{align*}
d^p_{\mathcal{W}_p}(\mu, \mu')\geq \sum_{k=0}^M \pi_0(x_k, x_1)d_N^p(x_k, x_1) \geq  \left( \sum_{k=0}^M \pi_0(x_k, x_1) \right) d_N^p(x_0, x_1)= a_0h_0^p,
\end{align*}
since for any $1 \leq k \leq M$, $d_1^p(x_k, x_1) \geq h >h_0$.
On the other hand, since $\mu'$ is obtained by shifting a weight $a_0$ from $x_1$ to $x_0$, we have
$$
d^p_{\mathcal{W}_p}(\mu, \mu') \leq a_0 d_N^p(x_0, x_1)=a_0 h_0^p.
$$
Thus we have $d_{\mathcal{W}_p}(\mu, \mu')=a_0^{\frac{1}{p}}h_0$. The proof of the other cases is similar. 

We also show that, for any measure $\xi \in \mathcal{W}_p(\Rn, d_N)$, if we have $d_{\mathcal{W}_p}(\mu, \xi)= a^{\frac{1}{p}}h_0$ and $P_{H\#}(\xi)=P_{H\#}(\mu')$, $P_{L\#}(\xi)=P_{L\#}(\mu')$ and $\text{supp}(\xi) \in \bigcup_{i=0}^M (H+x_i)$, then $\xi=\mu'$.  

To see this, consider a measure $ \xi \in \mathcal{W}_p(\Rn, d_N)$ such that we get $P_{H\#}(\xi)=P_{H\#}(\mu')$, $P_{L\#}(\xi)=P_{L\#}(\mu')$ and $\text{supp}(\xi)\in  \bigcup_{i=0}^M (H+x_i)$. By the same argument as above we can write $\xi$ as 
$$
\xi = \sum_{k=0}^{M}\sum_{k'=1}^M \tilde{b}_{ij}\delta_{x_{k,k'}} 
$$  
with 
$$
\sum_{k'=1}^M \tilde{b}_{0k'}=a_0, \quad \sum_{k'=1}^M\tilde{b}_{1k'}=a_1-a_0, \quad \sum_{k'=1}^M b_{kk'}=a_k \text{ for } 2 \leq k \leq M, \quad \sum_{k=0}^M\tilde{b}_{kk'}=a_{k'}.
$$
Similar to above we get that, for $\pi_0$ an optimal transport plan between $\mu$ and $\xi$, since $\pi_0(x_{1}, x_1) \leq \tilde{b}_{11} \leq a_1-a_0$, we have that $\sum_{k=0}^M\sum_{k'=1}^M\pi_0(x_{k,k'}, x_1)-\pi_0(x_1, x_1)\geq a_0$, and thus
\begin{align*}
d^p_{\mathcal{W}_p}(\mu, \xi)\geq \sum_{k=0}^M\sum_{k'=1}^M \pi_0(x_1, x_{k,k'})d_N^p(x_1, x_{k,k'}) \geq  a_0 d_N^p(x_1, x_0)= a_0h_0^p.
\end{align*} 
To have equality, we need that $\pi_0(x_{k,k'}, x_{k''})=0$ unless $k'' = 1$ or $k=k'=k''$. Since $d_N(x_0, x_1)<d_N(x_{k,k'}, x_1)$ unless $(k,k')=(0,1)$, we also need that $\pi_0(x_{k,k'}, x_{1})=0$ for $(k,k')\neq (0,1)$ or $k=k'=1$. Thus the support of $\pi_0$ is exactly $\{ (x_k, x_k) \} \cup \{(x_0, x_1)\}$. From the conditions on the weights we can thus conclude that $\xi=\mu$.

To get the desired contradiction, we notice that, while $\nu_1' \neq \nu_2'$, 
$$d_{\mathcal{W}_p}(\Phi(\mu), \nu_1')=d_{\mathcal{W}_p}(\Phi(\mu), \nu_2')=a_0^{\frac{1}{p}}h_0.$$
 Since $\Phi^{-1}$ is an isometry, we thus have that $d_{\mathcal{W}_p}(\mu,\Phi^{-1}( \nu_1'))=d_{\mathcal{W}_p}(\mu,\Phi^{-1}( \nu_2'))=a_0^{\frac{1}{p}}h_0$, with $\Phi^{-1}( \nu_1') \neq \Phi^{-1}( \nu_2')$. By Lemma \ref{lem: projection commutates}, we have that $P_{H \#}(\Phi^{-1}( \nu_1'))=P_{H \#}( \nu_2')=P_{H \#}(\mu')$ and $P_{L \#}(\Phi^{-1}( \nu_1'))=P_{L \#}( \nu_2')=P_{L\#}(\mu')$. Lemma \ref{lem: restrict_image} shows that 
 $$\text{supp}(\nu_1'), \text{ supp}(\nu_2') \in \bigcup_{i=0}^M (H+x_i).$$ But then, by the previous remark,  $\Phi^{-1}( \nu_1')=\mu'=\Phi^{-1}( \nu_2')$, giving the contradiction. Therefore we have that indeed $\Phi(\mu) \in \mathcal{F}$. 

To finish the proof of the proposition, we combine the three properties of $\mathcal{F}$ to show that, if $\mu \in \Wp(\Rn, d_N)$, then $\Phi(\mu)=\mu$. By Lemma \ref{lem: projection commutates}, we have that, 
$$
P_{H \#}(\Phi(\mu))=\Phi(P_{H \#}(\mu))=P_{H \#}(\mu),
$$
since $\Phi$ is assumed to be invariant on measures supported on $H$. By the injectivity of the projection map from the set $\mathcal{F}$, we have that $\Phi(\mu)=\mu$ whenever $\mu\in \mathcal{F}$. By the density of $\mathcal{F}$ we can extend the same result to any $\mu  \in \Wp(\Rn, d_N)$.
\end{proof}

%% file: mu_and_Phimu_images.tex
\begin{figure}
\centering
\begin{tikzpicture}[scale=0.6]
\draw[help lines, color=gray!30, dashed] (0,0) grid (9,7);
\draw[->,ultra thick] (0,0)--(9,0) node[right]{$L$};
\draw[->,ultra thick] (0,0)--(0,7) node[above]{$H$};

\draw[black, thin] (0,5) -- (9,5);
\draw[black, thin] (0,4) -- (9,4);
\draw[black, thin] (0,2) -- (9,2);

\draw[black, thin] (2,0) -- (2,7);
\draw[black, thin] (4,0) -- (4,7);
\draw[black, thin] (8,0) -- (8,7);

\filldraw[red] (2,2) circle (3pt) node[anchor=north east]{$x_1$};
\filldraw[red] (4,5) circle (3pt) node[anchor=north east]{$x_2$};
\filldraw[red] (8,4) circle (3pt) node[anchor=north east]{$x_3$};
\end{tikzpicture}
\begin{tikzpicture}[scale=0.6]
\draw[help lines, color=gray!30, dashed] (0,0) grid (9,7);
\draw[->,ultra thick] (0,0)--(9,0) node[right]{$L$};
\draw[->,ultra thick] (0,0)--(0,7) node[above]{$H$};

\draw[black, thin] (0,5) -- (9,5);
\draw[black, thin] (0,4) -- (9,4);
\draw[black, thin] (0,2) -- (9,2);

\draw[black, thin] (2,0) -- (2,7);
\draw[black, thin] (4,0) -- (4,7);
\draw[black, thin] (8,0) -- (8,7);

\filldraw[blue] (2,5) circle (2pt) node[anchor=north east]{$x_{1, 2}$};
\filldraw[blue] (4,2) circle (2pt) node[anchor=north east]{$x_{2, 1}$};
\filldraw[blue] (4,5) circle (2pt) node[anchor=north east]{$x_{2, 2}$};
\filldraw[blue] (2,2) circle (2pt) node[anchor=north east]{$x_{1, 1}$};
\filldraw[blue] (8,4) circle (3pt) node[anchor=north east]{$x_{3, 3}$};
\end{tikzpicture}
\caption{An example for a measure $\mu \in \mathcal{F}$ (left) and a possible image measure $\Phi(\mu)$ (right).}
\label{fig:M2}
\end{figure}
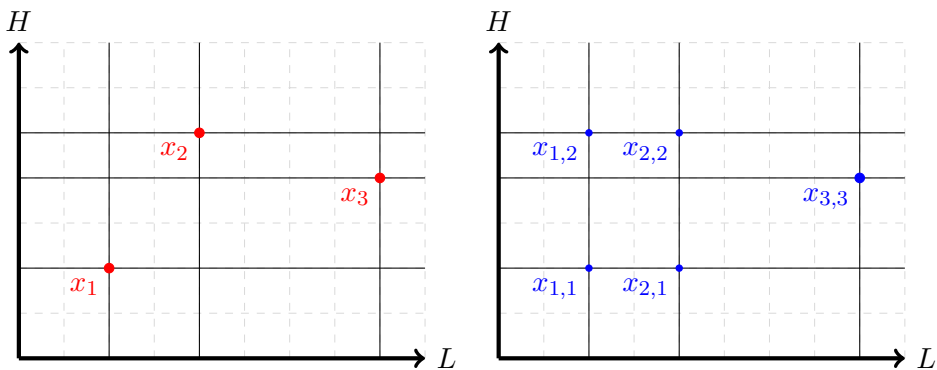

%% file: mu_and_nu_prime.tex
\begin{figure}
\centering
\begin{tikzpicture}[scale=0.6]
\draw[help lines, color=gray!30, dashed] (0,0) grid (9,7);
\draw[->,ultra thick] (0,0)--(9,0) node[right]{$L$};
\draw[->,ultra thick] (0,0)--(0,7) node[above]{$H$};

\draw[black, thin] (0,5) -- (9,5);
\draw[black, thin] (0,4) -- (9,4);
\draw[black, thin] (0,2) -- (9,2);

\draw[black, thin] (2,0) -- (2,7);
\draw[black, thin] (2.5,0) -- (2.5,7);
\draw[black, thin] (4,0) -- (4,7);
\draw[black, thin] (8,0) -- (8,7);

\filldraw[red] (2,2) circle (2pt) node[anchor=north east]{$x_1$};
\filldraw[red] (2.5,2) circle (1pt) node[anchor=north west]{$x_0$};
\filldraw[red] (4,5) circle (3pt) node[anchor=north east]{$x_2$};
\filldraw[red] (8,4) circle (3pt) node[anchor=north east]{$x_3$};

\end{tikzpicture}

\begin{tikzpicture}[scale=0.6]
\draw[help lines, color=gray!30, dashed] (0,0) grid (9,7);
\draw[->,ultra thick] (0,0)--(9,0) node[right]{$L$};
\draw[->,ultra thick] (0,0)--(0,7) node[above]{$H$};

\draw[black, thin] (0,5) -- (9,5);
\draw[black, thin] (0,4) -- (9,4);
\draw[black, thin] (0,2) -- (9,2);

\draw[black, thin] (2.5, 0) -- (2.5, 7);
\draw[black, thin] (2,0) -- (2,7);
\draw[black, thin] (4,0) -- (4,7);
\draw[black, thin] (8,0) -- (8,7);

\filldraw[blue] (2,5) circle (2pt) node[anchor=north east]{$x_{1, 2}$};
\filldraw[blue] (2,2) circle (2pt) node[anchor=north east]{$x_{1, 1}$};
\filldraw[blue] (2.5,2) circle (1pt) node[anchor=south west]{$x_{0, 1}$};
\filldraw[blue] (4,2) circle (2pt) node[anchor=south west]{$x_{2, 1}$};
\filldraw[blue] (4,5) circle (2pt) node[anchor=south west]{$x_{2, 2}$};
\filldraw[blue] (8,4) circle (3pt) node[anchor=north east]{$x_{3, 3}$};
\end{tikzpicture}
\begin{tikzpicture}[scale=0.6]
\draw[help lines, color=gray!30, dashed] (0,0) grid (9,7);
\draw[->,ultra thick] (0,0)--(9,0) node[right]{$L$};
\draw[->,ultra thick] (0,0)--(0,7) node[above]{$H$};

\draw[black, thin] (0,5) -- (9,5);
\draw[black, thin] (0,4) -- (9,4);
\draw[black, thin] (0,2) -- (9,2);

\draw[black, thin] (2.5, 0) -- (2.5, 7);
\draw[black, thin] (2,0) -- (2,7);
\draw[black, thin] (4,0) -- (4,7);
\draw[black, thin] (8,0) -- (8,7);

\filldraw[blue] (2,5) circle (2pt) node[anchor=north east]{$x_{1, 2}$};
\filldraw[blue] (2.5,5) circle (1pt) node[anchor=south west]{$x_{0, 2}$};
\filldraw[blue] (2,2) circle (2pt) node[anchor=north east]{$x_{1, 1}$};
\filldraw[blue] (4,2) circle (2pt) node[anchor=south west]{$x_{2, 1}$};
\filldraw[blue] (4,5) circle (2pt) node[anchor=south west]{$x_{2, 2}$};
\filldraw[blue] (8,4) circle (3pt) node[anchor=north east]{$x_{3, 3}$};
\end{tikzpicture}
\caption{The constructions $\mu'$ (up), $\nu_1'$ (down left) and $\nu_2'$ (down right).}
\label{fig:M1}
\end{figure}
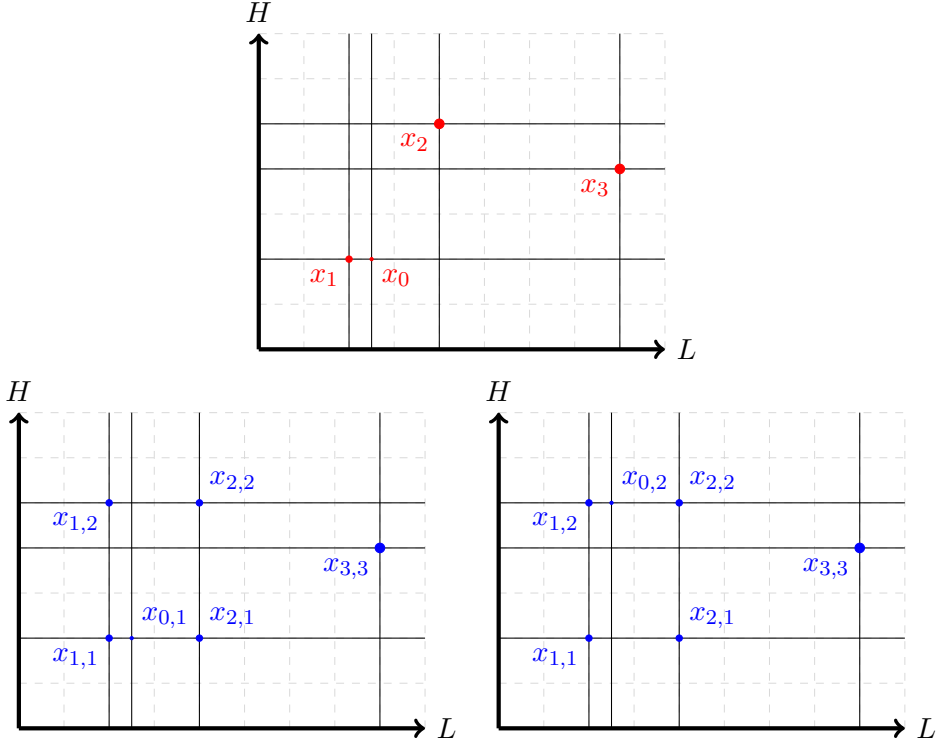

%% file: Rigid_continuous_C2.tex
\section{Isometric rigidity using $C^2$-differentiability}
\label{sec:potential}

In this section we prove Theorem \ref{thm1}, showing that the $C^2$-smoothness and the strict convexity of a norm is a sufficient condition for the rigidity of the Wasserstein space $\mathcal{W}_p(\Rn, d_N)$ when $p \neq 2$. We consider $p \geq 1,$ $p \neq 2$ fixed. We will work on the normed space $(\Rn,d_N)$, where the norm $N: \Rn \to \R_{+}$ is $C^2$-smooth on $\Rn \setminus \{0\}$, and $N$ is strictly convex. 
By the Corollaries \ref{corr:p>1 Dirac fix} and \ref{corr:p=1 Dirac fix}, for an isometry $\Phi:\Wp(\Rn,d_N) \to \Wp(\Rn,d_N)$, we can assume that $\Phi(\delta_x)=\delta_x$ for all $x\in\Rn$. 

Our goal is to prove that $\Phi(\mu)=\mu$ for any measure $\mu \in \mathcal{W}_p(\R^n, d_N)$ using the so-called \emph{potential functions} $\potmu$. For a measure $\mu$, we define $\potmu$ as follows:
\begin{equation}\label{eq:pot}
    \potmu\colon \Rn\to \R_+, \quad x\mapsto \dwp^p(\mu,\delta_x) = \int_{\Rn} N^p(x-y)~\dd\mu(y).
\end{equation}
Observe that $\potmu(x)=\potphimu(x)$ for all $x\in\Rn$. Indeed,
$$\potmu(x)=\dwp^p(\mu,\delta_x)=\dwp^p(\Phi(\mu),\Phi(\delta_x))=\dwp^p(\Phi(\mu),\delta_x)=\potphimu(x).$$
The question is whether $\potmu\equiv \mathcal{T}_{\nu}^{(p)}$ implies $\mu=\nu$? 
The answer in the general normed setting is no. 
To see an instructive example, consider the plane $\R^2$ equipped with the maximum norm $N_{\infty}(x_1,x_2)=\max\{|x_1|,|x_2|\}$. If we now take the measures 
\begin{align*}
\mu=\frac{1}{2}(\delta_{(0,1)} + \delta_{(0, -1)})\text{ and }\nu=\frac{1}{4}(\delta_{(0,1)} + \delta_{(0,-1)} + \delta_{(1, 0)} + \delta_{(-1,0)}),
\end{align*}
then the potential functions are given by 
$$
\mathcal{T}^{(1)}_{\mu}(x_1,x_2)=\frac{1}{2}(\max\{|x_1|, |x_2| + 1\} + \max\{|x_1|+1, |x_2|\}-1)=\mathcal{T}^{(1)}_{\nu}(x_1, x_2). 
$$
Thus, in a general normed setting, we cannot conclude from $\mathcal{T}^{(p)}_{\mu}(x)=\mathcal{T}^{(p)}_{\nu}(x)$ that $\mu=\nu$. However, we will now show that, if the norm $N$ is at least $C^2$-smooth, then $\mathcal{T}^{(p)}_{\mu}(x)=\mathcal{T}^{(p)}_{\nu}(x)$ implies $\mu=\nu$. 

The proof of rigidity is split into two parts. We start by giving a direct proof in the case $p \in [1,2)$. The second part will look at $p \in (2, \infty)$, where we show that measures supported on certain subspace "remain" on these subspaces after an isometry, and we finish with an induction argument using the proposition of the previous section. 

\subsection{First case: $p<2$}

We fix $p \in [1, 2)$, and we want to prove the following equality:
$$\lim_{\substack{h \to 0 \\ h \neq 0}}\frac{\mathcal{T}^{(p)}_{\mu}(x+h)-2\mathcal{T}^{(p)}_{\mu}(x)+\mathcal{T}^{(p)}_{\mu}(x-h)}{2 N^p(h)} = \mu(\{x\}).$$
As a first step, we show that
\begin{equation}
\label{eq: char0}
\lim_{\substack{h \to 0 \\ h \neq 0}}\frac{N^p(x+h)-2N^p(x)+N^p(x-h)}{2 N^p(h)}=\mathds{1}_{\{0\}}(x).
\end{equation}
 If $x=0$, the limit is trivially equal to $1$. For $x \neq 0$, the function $y\to N^p(y)$ is twice differentiable at $x$, and Taylor's expansion gives
$N^p(x \pm h)=N^p(x)+\langle\nabla N^p(x),\pm h\rangle+O(\|h\|^2)$
where $\|\cdot\|$ denotes the Euclidean norm. From here, we see that $N^p(x+h)-2N^p(x)+N^p(x-h)=O(\|h\|^2)$, and thus
$$\lim_{\substack{h \to 0 \\ h \neq 0}}\frac{N^p(x+h)-2N^p(x)+N^p(x-h)}{2 N^p(h)}=\lim_{\substack{h \to 0 \\ h \neq 0}}\frac{O(\|h\|^2)}{2 N^p(h)}=0,$$
because $N$ and $\|\cdot\|$ are equivalent norms (any two norms are bi-Lipschitz equivalent over a finite-dimensional vector space). Now integrating over $\Rn$ against $\mu$ we have
$$\mathcal{T}^{(p)}_{\mu}(x \pm h)=\int_{\Rn}N^p(x \pm h-y)~\dd\mu(y),$$
and thus
$$\lim_{\substack{h \to 0 \\ h \neq 0}}\frac{\mathcal{T}^{(p)}_{\mu}(x+h)-2\mathcal{T}^{(p)}_{\mu}(x)+\mathcal{T}^{(p)}_{\mu}(x-h)}{2 N^p(h)}$$
can be written as
\begin{equation}
\label{eq:tolebesgue}
\lim_{\substack{h \to 0 \\ h \neq 0}}\int_{\Rn}\frac{N^p(x+h-y)-2N^p(x-y)+N^p(x-h-y)}{2 N^p(h)}~\dd\mu(y).
\end{equation}

We want to show now that the function $G: \Rn \times \Rn \to \R$ given by  
\begin{equation*}
G(x,h)=
\begin{cases}
\frac{N^p(x+h)-2N^p(x) + N^p(x-h)}{2N^p(h)}& \text{ if }h \neq 0  \\
\mathds{1}_{\{0\}}(x) & \text{ if }h =0
\end{cases}
\end{equation*}
is bounded. To see this, we first notice that $G(\lambda x, \lambda h)=G(x, h)$ for any $\lambda\in \R$, thus we only need to prove that $G$ is bounded on the set $E \times E$, with $E=\{y \in \Rn \text{ such that }N(y) \leq 1 \}$. When $x=0$, then $G(0,h)=1$ is trivially bounded. When $x \neq 0$, we can use Taylor's expansion to write 
$$
G(x, h)=\frac{O(\|h\|^2)}{2N^p(h)}.
$$
Since on $E \times E$ we have $N(h)\leq 1$, and since $N$ and $\|\cdot \|$ are equivalent norms, $G$ is bounded on $E \times E$, and thus is bounded on $\Rn \times \Rn$.

Since $G$ is bounded, we can use the Lebesgue-dominated convergence theorem to interchange the limit and the integral of \eqref{eq:tolebesgue}, and thus we get $$\lim_{\substack{h \to 0 \\ h \neq 0}}\frac{\mathcal{T}^{(p)}_{\mu}(x+h)-2\mathcal{T}^{(p)}_{\mu}(x)+\mathcal{T}^{(p)}_{\mu}(x-h)}{2 N^p(h)}=\int_{\Rn}\mathds{1}_{\{0\}}(x-y)~\dd\mu(y)=\mu(\{x\}).$$

To prove rigidity, consider a finitely supported measure $\mu$. Then, since $\mathcal{T}^{(p)}_{\mu}=\mathcal{T}^{(p)}_{\Phi(\mu)}$, we have that, for any $x_i \in \supp(\mu)$, 
\begin{align*}
\mu(\{x_i\})=&\lim_{\substack{h \to 0 \\ h \neq 0}}\frac{\mathcal{T}^{(p)}_{\mu}(x+h)-2\mathcal{T}^{(p)}_{\mu}(x)+\mathcal{T}^{(p)}_{\mu}(x-h)}{2 N(h)} \\
=&\lim_{\substack{h \to 0 \\ h \neq 0}}\frac{\mathcal{T}^{(p)}_{\Phi(\mu)}(x+h)-2\mathcal{T}^{(p)}_{\Phi(\mu)}(x)+\mathcal{T}^{(p)}_{\Phi(\mu)}(x-h)}{2 N(h)}=\Phi(\mu)(\{x_i\}),
\end{align*} and thus $\Phi(\mu)=\mu$. Since finitely supported measures are dense in $\mathcal{W}_p(\Rn, d_N)$, we have that $\Phi(\mu)=\mu$ for any $\mu \in \mathcal{W}_p(\Rn, d_N)$, proving Theorem \ref{thm1} for the case $p\in [1, 2)$. 

\subsection{Second case: $p>2$}

While we could, similar to the method in \cite{GTV2}, adapt the previous idea for higher $p$, this approach presents two problems. First, it would require a Taylor expansion of the order of $\lceil p \rceil$, while the norms we consider are only $C^2$-smooth. Second, this method does not work when $p$ is an even integer.

Instead, we show in the next proposition that, if we restrict the isometry to measures supported on a certain subspace, the image will be measures supported on the same subspace. The rigidity of the Wasserstein space then follows by an induction on the dimension $n$, using the result of the previous section.  

\begin{proposition}
\label{prop: Hessian_restrict}
If $p>2$ and $N: \Rn \to \R_{+}$ is a $C^{2}$-smooth norm, then there exists a proper linear subspace $L \subset \Rn$ such that for any $x_0 \in \Rn$, if $\mu$ is a measure supported on $x_0 + L$, then for any isometry $\Phi$ of the Wasserstein space $\mathcal{W}_p(\Rn, d_N)$ fixing Dirac masses, $\Phi(\mu)$ will also be supported on $x_0 +L$.
\end{proposition}
Since $N$ is a convex function, we have that the matrix $(\text{Hess}N^p)(x)$ is positive semi-definite for any $x \in \Rn \setminus \{0\}$. We start by showing the following proposition. 
\begin{proposition}
\label{prop: non-constant derivative}
For any $N: \Rn \to \R_+$ a $C^2$-smooth norm, there exists $v_1, v_2 \in S^{n-1}$ such that 
$$
\min_{x \in \Rn \setminus \{0\}}v_2^{\top} (\text{Hess} N^p)(x) v_1=0
$$ 
and the function $x \to v_2^{\top} (\text{Hess} N^p)(x)v_1$ is non-negative and non-constant on $\Rn \setminus \{0\}$.
\end{proposition}
Here $S^{n-1}$ is the usual unit sphere with respect to the Euclidean metric, and we consider the unit sphere of the norm $S_N^{n-1}=\{x \in \Rn | N(x)=1 \}$.  

\begin{proof}

If there exists $v_1 \in S^{n-1}$ and $x_0 \in S_N^{n-1}$ such that $v_1^{\top} (\text{Hess}N^p)(x_0) v_1=0$, this proves our claim. Indeed, using that $N^p(tv)=t^pN^p(v)$ when $t>0$, we can easily calculate  that 
$v_1^{\top} (\text{Hess}N^p)(v_1) v_1=p(p-1)>0$ and, due to positive semi-definiteness, that $$v_1^{\top} (\text{Hess}N^p)(x_0) v_1 \geq 0$$ for all $x \in S^{n-1}_N$. Thus we assume that   
\begin{equation}
\label{assumption:1}
v_1^{\top} (\text{Hess}N^p)(x) v_1>0
\end{equation} for all $x \in S_N^{n-1}$ and all $v_1 \in S^{n-1}$. 

Then we have in particular that $(-v_1)^{\top} (\text{Hess}N^p)(x) v_1<0$. If $v^1, \dots, v^n$ is a linearly independent basis of $\Rn$ with $v^1=v_1$, consider for $2 \leq i \leq n$ the curves $\gamma_i: [0,1] \to S^{n-1}$ given by $\gamma_i(t)=\tilde{n}((-1+2t)v^1 + (1-|2t-1|)v^i)$, where $\tilde{n}$ is the normalization operator $\tilde{n}(x)=\frac{x}{\|x\|_E}$, and $\| \cdot \|_E$ is the usual Euclidean norm.

We define the map $H_i: [0, 1] \times S_N^{n-1} \to \R$ (which depends implicitly on the choice of $v_1$) as
$$
H_i(t, x)=\gamma_i(t)^{\top}(\text{Hess}N^p)(x) v_1,
$$
define for any $x \in S_N^{n-1}$ the set 
$A^i_x=\{ t \in [0, 1] | H_i(t,x)=0 \}$
and we take $t^i_x = \sup(A^i_x)$. Then $H_i(t, x)>0$ for any $t \in [0,1]$ with $t>t^i_x$. We also set $t_i=\sup\{t^i_x | x \in S_N^{n-1}\}$. Since $H_i$ is continuous on a compact set, the preimage $H^{-1}(\{0\})$ is closed and compact, and thus $t_i \in (0,1)$, and there exists $x_i \in S_N^{n-1}$ such that $t_i=t_{x_i}$. Then $H_i(t_i, x_i)=0$, and for any $x \in S_N^{n-1}$, $H_i(t_i, x)\geq 0$. 

If for some $v_1 \in S^{n-1}$ and for some $2 \leq i \leq n$ the function $H_i(t_i, \cdot)$ is not constant on $S_N^{n-1}$, choosing $v_2=\gamma_i(t_i)$ proves the lemma. Thus we assume by contradiction that this never holds, i.e. for any $v_1 \in S^{n-1}$ and any $2 \leq i \leq n$, $H_i(t_i, x)=0$ for any $x\in S_N^{n-1}$.

Since by assumption \eqref{assumption:1} we have $\text{Hess}N^p(x) v_1 \neq 0$, the equality $H_i(t_i, x)=0$ implies that, for any $x \in S_N^{n-1}$, $\text{Hess}N^p(x) v_1$ is a vector in the hyperplane orthogonal to $\gamma_i(t_i)$. As by construction the collection $\{ \gamma_i(t_i) \}_{i=2}^n$ is a set of $n-1$ linearly independent vectors, if $\text{Hess}N^p(x) v_1$ is orthogonal to every vector in the collection, there exists $w_{v_1} \in S^{n-1}$ (which is orthogonal to all $\gamma_i(t_i)$) such that $\text{Hess}N^p(x) v_1=\lambda_{v_1}(x)w_{v_1}$ for some function $\lambda_{v_1}: S_N^{n-1} \to \R$. 

Choosing $v_1=e_i$ a canonical vector, we thus have that $\partial_{ij}N^p(x)=\lambda_{e_i}(x) (w_{e_i})_j$; since assumption \eqref{assumption:1} guarantees that $\partial_i^2N^p(x)>0$ for all $x\in S_N^{n-1}$, $(w_{e_i})_i\neq 0$ and we get 
\begin{equation}
\partial_{ij}N^p(x)= \frac{(w_{e_i})_j}{(w_{e_i})_i}\partial_{i}^2N^p(x).
\end{equation}
Since $N^p$ is $p-$homogeneous, this still holds true for any $x \in \Rn \setminus \{0\}$. 
If $\partial_{ij}N^p(x)=0$ for all $i,j \leq n$, $i \neq j$, then using partial integration we show that we can write $N^p(x)=\sum_{i=1}^n a_i(x_i)$ for some functions $a_i: \R \to \R$. As $\partial_{ij}N^p$ is $(p-2)$-homogeneous, we see that (for example) $(\partial_1^2 a_1)(0)=0$ and $\partial_1^2 N^p(e_2)=0$ contradicting assumption \eqref{assumption:1}.

 Thus we take $i \neq  j$ such that $\partial_{ij}N^p(x)\neq 0$. Then 
\begin{equation}
\label{assumption2}
\frac{(w_{e_i})_j}{(w_{e_i})_i}\partial_{i}^2N^p(x)= \partial_{ij}N^p(x)= \frac{(w_{e_j})_i}{(w_{e_j})_j}\partial_{j}^2N^p(x).
\end{equation}
We set $c=\frac{(w_{e_i})_j}{(w_{e_i})_i}$ and $d=\frac{(w_{e_j})_i}{(w_{e_j})_j}$. Using $v_1=\tilde{n}(\frac{1}{\sqrt{c}}e_i -\frac{1}{\sqrt{d}}e_j)$ (where $\tilde{n}$ is the normalization operator), assumption \eqref{assumption:1} gives the condition that $cd>1$. 

We first assume that $c,d>0$ (and thus $\partial_{ij}N^p(x)>0$) and consider the vector $v_1(\vartheta)=\cos (\vartheta)e_i -  \sin(\vartheta) e_j$ for $\vartheta \in [0, 2 \pi)$. Then \eqref{assumption:1} implies that 
$$
\cos^2(\vartheta) \partial_{i}^2N^p(x)-2 \sin(\vartheta)\cos(\vartheta) \partial_{ij}N^p(x)+\sin^2(\vartheta) \partial_{j}^2N^p(x)>0.
$$  
Using \eqref{assumption2}, we rewrite this as 
$$
\left( \frac{1}{c}\cos^2(\vartheta) -2 \sin(\vartheta)\cos(\vartheta) +\frac{1}{d}\sin^2(\vartheta) \right) \partial_{ij}N^p(x)>0 \\
$$
which implies $$
d\cos^2(\vartheta) -2 cd\sin(\vartheta)\cos(\vartheta) +c\sin^2(\vartheta)>0.
$$
This is in turn equivalent to 
$$
\left(\sqrt{d}\cos(\vartheta) - \sqrt{c}\sin(\vartheta) \right)^2 - 2\left(cd - \sqrt{cd} \right) \sin(\vartheta)\cos(\vartheta) >0.
$$
By choosing $\vartheta \in (0, \frac{\pi}{2})$ such that $\tan (\vartheta)=\sqrt{\frac{d}{c}}$, the square of the previous equation vanishes, and since $\sin(\vartheta), \cos(\vartheta)>0$, we get the condition 
$$
cd - \sqrt{cd}<0,
$$
contradicting our assumption that $cd>1$. 

If $c,d<0$, we instead define  $v_1(\vartheta)=\cos (\vartheta)e_i + \sin(\vartheta) e_j$. By the same argument we then get 
$$
d\cos^2(\vartheta) +2 cd\sin(\vartheta)\cos(\vartheta) +c\sin^2(\vartheta)<0
$$ or 
$$
|d|\cos^2(\vartheta) -2 cd\sin(\vartheta)\cos(\vartheta) +|c|\sin^2(\vartheta)>0,
$$
which gives the same contradiction as in the previous case. 

\end{proof}
For $v_1,v_2$ given by Proposition \ref{prop: non-constant derivative}, we consider the set 
$$A= \{x \in \Rn \setminus \{0\} | v_2^{\top} (\text{Hess}N^p)(x)v_1=0 \},$$ and $A_0=A \cup \{0\}$. We also define the function $T: \Rn \setminus \{0\} \to \R$ as 
$$
T(x)=v_2^{\top}(\text{Hess}N^p)(x)v_1.
$$
Then, for $x\neq 0$, we have $T(x)=0$ if $x\in A$, and $T(x)>0$ if $x \notin A$. Since $\text{Hess}N^p$ is $(p-2)$-homogeneous, with $p>2$, we have that $T(0)=0$. For a general measure $\mu$, we define $\mathbb{T}_{\mu}: \Rn \to \R$ as 
$$
\mathbb{T}_{\mu}(x)=\int_{\Rn}T(x-y) d\mu(y).
$$ Using the Lebesgue convergence theorem, we can show that 
$$
\mathbb{T}_{\mu}(x)=v_2^{\top}\text{Hess}\left(\int_{\Rn}N^p(x-y)d\mu(y)\right) v_1=v_2^{\top}\text{Hess} \left( \mathcal{T}^{(p)}_{\mu}(x) \right) v_1,
$$ and thus $\mathbb{T}_{\mu}(x)=\mathbb{T}_{\Phi(\mu)}(x)$. 

We now have the following Lemma. 
\begin{lemma}
\label{lem: support}
Let $\mu \in \mathcal{W}_p(\Rn, d_N)$ be a measure. If there exists a point $x_0 \in \Rn$ such that $\mathbb{T}_{\mu}(x_0)=0$, then $\mu$ is supported on the set of points $\{x_0 + A_0\}.$ Furthermore, if $\mu$ is a measure supported on an affine subset $x_0 + L \subset x_0 + A_0$ for a proper linear subset $L \subset \Rn$, then $\mathbb{T}_{\mu}(x)=0$ for all $x \in x_0 + L$.
\end{lemma}

\begin{proof}
We start by showing the first part of the statement. Assume that $\mathbb{T}_{\mu}(x_0)=0$, but $\mu$ is not supported on $x_0 + A_0$, i.e. there exists $x_0' \in \supp(\mu)$ such that $x_0' \notin x_0 + A_0$. By the definition of the support of a measure, for any ball $B_N(\varepsilon, x_0')$ around $x_0'$ with radius $\varepsilon>0$, we have that $\mu(B_N(\varepsilon, x_0'))>0$. 
 
We choose $\varepsilon>0$ such that $\{ x_0 + A_0 \} \cap B_N(\varepsilon, x_0')= \emptyset$ (since $N^p$ is $C^2$-continuous, the set $A_0$ is closed, and such an $\varepsilon$ exists). 
Then, for any $y \in \overline{B_N}(\varepsilon/2, x_0')=:B_1$, we have $y \notin x_0 + A_0$ and 
\begin{equation}
\label{eq:Tpi_limit}
T(x_0-y)>0.
\end{equation}
Since $B_1$ is compact, \eqref{eq:Tpi_limit} attains its minimum at a point $y_0 \in B_1$, at which we still have $T(x_0- y_0)>0$. Thus, since $T(x)$ is a non-negative function, we have that
$$
\mathbb{T}_{\mu}(x_0) \geq \int_{B_1}\ T(x_0 -y)d\mu(y) \geq T(x_0-y_0)\mu(B_1)>0.
$$
Therefore, if there exists $x_0 \in \Rn $ with  $\mathbb{T}_{\mu}(x_0)=0$, then $\mu$ is supported on the set $\{x_0 + A_0 \}$. 

For the second part of the lemma, consider a measure $\mu$ supported on $x_0 + L$. Then, if $x, y \in x_0 + L$, since by linearity $x-y \in A_0$, we have that $T(x-y)=0$ and 
$$
\mathbb{T}_{\mu}(x)=\int_{x_0 + L}T(x-y)d\mu(y)=0
$$
for any $x \in x_0 +  L$.
\end{proof}

\begin{lemma}
\label{lem: support_2}
Consider for any $x_0 \in \Rn$ a proper maximal affine subset $x_0 + L \subset \{x_0 +A_0 \}$, in the sense that, for every $y \notin L$, we have that $\text{span}(\{L, y\}) \nsubseteq A_0$. If $\mu\in \mathcal{W}_p(\Rn, d_N)$ is a measure supported on $x_0 + L$ and $\Phi: \mathcal{W}_p(\Rn, d_N) \to \mathcal{W}_p(\Rn, d_N)$ is an isometry fixing Dirac masses, then $\Phi(\mu)$ is also supported on $x_0 + L$. 
\end{lemma}

\begin{proof}
Consider $\mu$ a measure supported on $x_0 + L$. Then, by the previous Lemma, $\mathbb{T}_{\mu}(x)=0$ for all $x \in x_0 +L$. By the same Lemma, since we have $\mathbb{T}_{\mu}(x)=\mathbb{T}_{\Phi(\mu)}(x)$, we get that $\Phi(\mu)$ is supported on the set $\{x + A_0 \}$ for any $x \in x_0+ L$. 

Assume for contradiction that there exists $x_1 \in \supp(\Phi(\mu))$ such that $x_1 \notin x_0 + L$. Then, for any $x \in x_0 + L$, we still have that $x_1-x \in A_0$. Consider $y \in \text{span}(L, x_1-x_0)$, $y \notin L$ and write $ y=\lambda_1 (x_1-x_0) - \lambda_2 x_2$ for $\lambda_0, \lambda_1 \in \R$, $\lambda_0\neq 0$ and $x_2 \in L$.  We can rewrite this as $y = \lambda_0 ((x_1-x_0) -(\frac{\lambda_1}{\lambda_0} x_1))$, and thus $y \in A_0$. But then $x_0 + \text{span}(\{L, x_1-x_0\}) \subset \{x_0 +A_0\}$, contradicting the maximality of the subset $L$.   

\end{proof}

Notice that we can always find a proper maximal subset. Indeed, by Lemma \ref{prop: non-constant derivative}, $A_0$ is a proper subset of $\Rn$, and thus a maximal set is also proper. Furthermore, since $A$ is not empty, we have $L_1= x_0 + tx_1\subset x_0 +A_0$ for any $x_1 \in A$. Then, if $L_1$ is not maximal, there exists $y \notin L_1$  such that $\text{span}(\{L_1, y\}) \subset x_0 +A_0$, and we consider the set $L=\text{span}(\{L_1, y \})$. If $L$ is not maximal, we proceed the same way, until we find a set $L' $ that is maximal. 

Therefore this lemma also proves Proposition \ref{prop: Hessian_restrict}. 

\begin{proof}[Proof of Theorem \ref{thm1} if $p>2$] Using Proposition \ref{prop: Hessian_restrict}, we now have  that, for any norm $N$, if $p> 2$, there exists a proper linear subspace $L$  such that for every $x_0 \in \Rn$, if a measure $\nu$ is supported on $L_{x_0}:=x_0 + L$, then $\Phi(\nu)$ is also supported on $L_{x_0}$. Note that for $p=2$, this is in general not true, as the example in Figure \ref{fig:M3} and Remark \ref{rem: generap_p=2} will show. 

 We first look at the case $n=2$. Since $L$ is proper, $L$ is a line and $(L_{x_0}, d_{N, L_{x_0}})$ is isometric to $(\R, |\cdot|)$, thus we can use \cite{GTV1} to assume that for any measure $\nu$ supported on $L_{x_0}$, $\Phi(\nu)=\nu$. By Lemma \ref{lem: Chebychev}, the set $P_{L}^{-1}(0)$ is again a line, and $P_{L}^{-1}(0)=H$, the linear subspace given by Lemma \ref{lem: restrict_image}. By Lemma \ref{lem: projection commutates}, we know that $P_{L\#}(\mu)=P_{L\#}(\Phi(\mu))$, which implies in particular that for every measure $\nu$ supported on $H$, we have that $P_{L\#}\nu=\delta_{0}$, which implies that $\Phi(\nu)$ is also supported on $H$, and we can again assume that $\Phi(\nu)=\nu$ on when $\nu$ is supported on $H$. Since the norm $N$ is strictly convex, it projects uniquely onto $L$ and $H$, and we can apply Proposition \ref{prop: Hyper implies hyper plus}. Then we have rigidity of the Wasserstein space $\mathcal{W}_p(\R^2, d_N)$ for $p>2$.

We now make the induction assumption on $n$, that for any $n \leq n_0$, the space $\mathcal{W}_p(\Rn, d_N)$ is rigid for $p \neq 2$ when $N$ is a $C^2$ strictly convex norm, and we consider the Wasserstein space $\mathcal{W}_p(\R^{n_0+1}, d_N)$. 

For $L$ given by Proposition \ref{prop: Hessian_restrict}, since $L$ is a proper linear subspace of $\R^{n_0+1}$, $(L_{x_0}, d_{N, L_{x_0}})$ is isometric to $(\Rn, d_{N, L_{x_0}})$ for some $n < n_0+1$. Since for any measure $\nu$ supported on $L_{x_0}$ we have that $\Phi(\nu)$ is also supported on $L_{x_0}$, we can use the induction assumption to assume that $\Phi(\nu)=\nu$ for any measure $\nu$ supported on $L_{x_0}$. 

We take the linear subspace $H \subset P_{L}^{-1}(0)$ given by Lemma \ref{lem: restrict_image}. Then, if $\nu$ is a measure supported on $H$, again by Lemma \ref{lem: restrict_image} we have that $\text{supp}(\Phi(\nu))\subset H + \text{supp}(\nu)=H$, and as above we can use the induction argument to assume that $\Phi(\nu)=\nu$ whenever $\nu$ is supported on $H$. Then, by Proposition \ref{prop: Hyper implies hyper plus} we have rigidity of the space $\mathcal{W}_p(\R^{n_0+1}, d_N)$.
\end{proof}

\subsection{Special case: $p=2$, $N=l_q$, $q>2$}

Note that in the proof of Theorem \ref{thm1}, we only use in a few places the fact that $p\neq 2$. In particular, we used in Proposition \ref{prop: non-constant derivative} that $\partial_{ij}N^p$ is $(p-2)-$homogeneous to show that $\partial_i^2N^p(e_2)=0$ if the mixed derivatives all vanish. In the definition of $T$, the $(p-2)-$homogeneity again allows us to define $T(0)$. Finally, in the proof of Theorem \ref{thm1}, we use that in \cite{GTV1}, the authors showed rigidity of the Wasserstein space over the real line when $p \neq 2$. 

\begin{remark}
\label{rem: generap_p=2}
In general, we know that these limitations are strict, since we have for $p=2$ examples of non-rigid Wasserstein spaces $\mathcal{W}_2(\Rn, d_N)$, for instance if the norm $N$ is the standard Euclidean norm.
\end{remark}
For example, in the case of $\R^2$, we can take the shape-preserving isometry which rotates measures around their center of mass with an angle of $\frac{\pi}{4}$, according to the construction given in Proposition 6.1 of Kloeckner \cite{K}. Then, even if this isometry preserves Dirac masses, any measure supported on a line will be sent to a measure supported on a different line, namely on the line rotated by $\pi/4$, as seen in Figure \ref{fig:M3}. Thus we cannot even hope to find a replacement of Proposition \ref{prop: Hessian_restrict} for general $C^2$-norms. 

Nevertheless, in the case of the $l_q$ norms with $q>2$, we can adapt the argument of the proof of Theorem \ref{thm1} to show the rigidity of the Wasserstein space $\mathcal{W}_2(\Rn, d_q)$. Indeed, for the $l_q$ norm given by  
$$
N_q(x)=\left(\sum_{i=1}^n |x_i|^q \right)^{\frac{1}{q}},
$$ since $q>2$, we have that $N_q$ is $C^2$-smooth and the second derivatives are given by 
$$
\partial_i^2N^2_q(x)=2(2-q)(N(x))^{2-2q} |x_i|^{2q-2} + 2 (q-1)(N(x))^{2-q}|x_i|^{q-2}.
$$

\input{tikztry.tex} 
In particular, we have $\partial_i^2N^2_q(x) \geq 0$ for $x \in \Rn \setminus \{0\}$, with equality if and only if $x_i=0$. Thus, instead of using Proposition \ref{prop: non-constant derivative} to define $T$, we can simply choose $i \leq n$, and set $T(x)=\partial_i^2N^p(x)$, which has the property that $T(x)=0$ if $x_i=0$, $x \neq 0$ and $ T(x)>0$ for $x_i \neq 0$. While $T$ is not necessarily defined at $x=0$, the proofs of Lemmata \ref{lem: support} and \ref{lem: support_2} still hold for measures that are absolutely continuous with respect to the Lebesgue measure of a subspace $L$ of dimension $k \geq 1$. Indeed, we then have that $\int_{\Rn}T(x-y)d\mu(y)=\int_{\Rn \setminus \{x\}}T(x-y)d\mu(y)$, which is well defined. Thus Proposition \ref{prop: Hessian_restrict} holds for measures that are absolutely continuous on subspaces; since a hypersurface of $\Rn$ is closed and since absolutely continuous measures are dense in  $\mathcal{W}_2(\Rn, d_q)$, we can analogously to Proposition \ref{prop: Hessian_restrict} state the following Proposition:
\begin{proposition}
\label{prop: restrict_lq}
Consider $q >2$, and $d_q$ the metric induced by the $l_q$ norm. For any $i \leq n$, consider the hypersurface $L_i$ given by $L_i = \{x \in \Rn | x_i=0\}$. If $\mu$ is a measure supported on $x_0 + L_i$, then for any isometry $\Phi$ of the Wasserstein space $\mathcal{W}_2(\Rn, d_q)$ fixing Dirac masses, $\Phi(\mu)$ will also be supported on $x_0 +L_i$.  
\end{proposition}

To finish the proof of Theorem \ref{thm2}, we want to again use an induction argument and Proposition \ref{prop: Hyper implies hyper plus}; however, we first need to show that, if an isometry $\Phi$ of the Wasserstein space fixes Dirac masses and globally fixes measures supported on a line, then it acts as the identity on measures supported on the line. 

To see this, consider in $\Rn$ the line $L_1=\{t e_1 | t \in \R \}$ (the other cases can be handled similarly). If an isometry $\Phi: \mathcal{W}_2(\Rn, d_q) \to \mathcal{W}_2(\Rn, d_q)$ globally preserves measures supported on $L_1$, since $(L_1, d_{q,L_1})$ is isometric to $(\R, | \cdot |)$, we can use the result from Kloeckner (Lemma 5.2 \cite{K}) to characterize $\Phi$. Indeed, we recall Kloeckner's notation  
$$
\mu=\mu(x, \sigma, p)=\frac{e^{-p}}{e^{-p} +e^p}\delta_{(x-\sigma e^p)e_1} \frac{e^{p}}{e^{-p} +e^p}\delta_{(x+\sigma e^{-p})e_1}
$$ to represent measures supported on two points on $L_1$, then Lemma 5.2 of \cite{K} gives us that an isometry which fixes Dirac masses and globally fixes $L_1$ acts on measures supported on two points of $L_1$ in the following way: 
$$
\Phi(\mu(x, \sigma, p))=\mu(x, \sigma, \varphi (p)),
$$ 
where $\varphi$  is an isometry of $(\R, |\cdot |)$. It is well known that the isometries of $(\R, |\cdot |)$ are given by $\varphi(x)= s x + t$ for some $t \in \R$ and $s \in \{-1, 1\}$. We will show that if an map $\Phi$ does not leave measures supported on 2 points of $L_1$ invariant (i.e. if $s =-1$ or $t \neq 0$), then $\Phi$ cannot be an isometry of the Wasserstein space $\mathcal{W}_2(\Rn, d_q)$. 

To start, assume by contradiction that the isometry $\Phi^t$ is given by $$\Phi^t(\mu(x, \sigma, p))=\mu(x, \sigma, p + t)$$ for some $t \neq 0$.

If we consider $\mu_0=\mu(0, 1, 0)=\frac{1}{2} \delta_{-e_1} + \frac{1}{2}\delta_{e_1}$, then 
$$
\Phi^t(\mu_0)=\frac{e^{-t}}{e^t + e^{-t}} \delta_{-e^te_1} +\frac{e^{t}}{e^t + e^{-t}} \delta_{e^{-t}e_1}. 
$$ 

We consider the Dirac mass $\nu=\delta_{e_2}$. Since $\Phi^t$ fixes Dirac masses, $\Phi^t(\nu)=\nu$, and 
$$
d^2_{\mathcal{W}_2}(\mu_1, \nu)=d^2_{\mathcal{W}_2}(\Phi^t(\mu_1), \Phi^t(\nu))=d^2_{\mathcal{W}_2}(\Phi^t(\mu_1), \nu). 
$$
We can explicitly calculate both distances to get the following equality 
\begin{align*}
 2^{\frac{2}{q}} =d^2_{\mathcal{W}_2}(\mu_1, \nu)=d^2_{\mathcal{W}_2}(\Phi^t(\mu_1), \nu)= \frac{e^{-t}}{e^t + e^{-t}}(e^{tq}+1)^{\frac{2}{q}} + \frac{e^{t}}{e^t + e^{-t}}(e^{-tq}+1)^{\frac{2}{q}}.
\end{align*}
We can rewrite this equality as 
\begin{align*}
2^{\frac{2}{q}}(e^t + e^{-t}) = (e^{tq/2}+e^{-tq/2})^{\frac{2}{q}} + (e^{-tq/2}+e^{tq/2})^{\frac{2}{q}}=2(e^{-tq/2}+e^{tq/2})^{\frac{2}{q}}
\end{align*} which is equivalent to 
\begin{align*}
\frac{1}{2}(e^{qt/2})^{\frac{2}{q}} + \frac{1}{2}(e^{-qt/2})^{\frac{q}{2}} =(\frac{1}{2}e^{-tq/2}+\frac{1}{2}e^{tq/2})^{\frac{2}{q}}.
\end{align*} 
Setting $A:=e^{\frac{tq}{2}}1$, the equation is written 
$$
\frac{1}{2}A^{\frac{2}{q}} + \left(\frac{1}{A} \right)^{\frac{2}{q}}=\left( \frac{1}{2}A + \frac{1}{2}\frac{1}{A} \right)^{\frac{2}{q}}. 
$$
But, since the function $s \to s^{\frac{2}{q}}$ is strictly concave when $q>2$, the equality can only hold if $A=\frac{1}{A}$, i.e. $e^{\frac{qt}{2}}=1$, which only holds for $t=0$, giving the contradiction.

Now assume by contradiction that the isometry $\Phi^*$ is given by $$\Phi^*: \mu(x,\sigma,p) \to \mu(x,\sigma,-p),$$ and consider $$\mu_1=\mu(0, \sqrt{2}, -\ln(2)/2)=\frac{2}{3}\delta_{-e_1}+\frac{1}{3}\delta_{2e_1}.$$ Then the image of $\mu_1$ is given by
$$
\Phi^*\ler{\mu_1}=\Phi^*\ler{\mu(0,\sqrt{2}, -\ln(2)/2)}=\mu(0,\sqrt{2}, \ln(2)/2)=\frac{1}{3}\delta_{-2e_1}+\frac{2}{3}\delta_{e_1}.
$$
We consider the Dirac mass $\nu=\delta_{e_1 + e_2}$ and, since we still have $\Phi^*(\nu)=\nu$, we have 
$$
3d^2_{\mathcal{W}_2}(\mu_1, \nu)=2\ler{2^q+1}^{2/q}+ 2^{2/q} = 2+ \ler{3^q+1}^{2/q}=  3d_{W_2}^2(\Phi^*(\mu_1),\nu)
$$
We can rewrite the left side as
\begin{align} \label{eq:mu-nu-dist}
    \frac{3}{4} d_{W_2}^2(\mu_1,\nu)=\ler{2^{-q/2}}^{2/q}\ler{2^q+2^0}^{2/q}+2^{2/q-2}
    =\ler{2^{q/2}+2^{-q/2}}^{2/q}+2^{2/q-2}
\end{align}
and  the right side as 
\begin{align} \label{eq:Phi-mu-nu-dist}
    \frac{3}{4} d_{W_2}^2(\Phi^*(\mu_1),\nu)
    =2^{-1}+\frac{1}{4}\ler{3^q+3^0}^{2/q}
    =2^{-1}+\ler{2^{-q}}^{2/q}\ler{3^q+3^0}^{2/q} \nonumber \\
    =2^{-1}+\ler{\ler{\frac{3}{2}}^q+\ler{\frac{1}{2}}^q}^{2/q}
    =2^{-1}+\ler{\ler{\frac{9}{4}}^{q/2}+\ler{\frac{1}{4}}^{q/2}}^{2/q}.
\end{align}
Let us compare the right-hand sides of \eqref{eq:mu-nu-dist} and \eqref{eq:Phi-mu-nu-dist}. Clearly, $2^{2/q-2}<2^{-1}$ if $q>2,$ and the key observation is that there exists a $\lambda \in (0,1)$ such that
\begin{align}
    2^{-1}=\lambda \cdot \frac{1}{4}+(1-\lambda)\cdot \frac{9}{4} \text{ and } 2=(1-\lambda) \cdot \frac{1}{4}+\lambda\cdot \frac{9}{4}.  
\end{align}
(The exact value of $\lambda$ is $\lambda=7/8,$ but this is not important.)
Therefore, 
\begin{align} \label{eq:convexity}
    2^{q/2}+2^{-q/2}
    =\ler{(1-\lambda) \cdot \frac{1}{4}+\lambda\cdot\frac{9}{4}}^{q/2}
    +\ler{\lambda \cdot \frac{1}{4}+(1-\lambda)\cdot\frac{9}{4}}^{q/2} \leq \nonumber \\
    \leq (1-\lambda) \cdot \ler{\frac{1}{4}}^{q/2} + \lambda\cdot\ler{\frac{9}{4}}^{q/2}
    +\lambda \cdot \ler{\frac{1}{4}}^{q/2}+(1-\lambda)\cdot\ler{\frac{9}{4}}^{q/2}
    =\ler{\frac{1}{4}}^{q/2}+\ler{\frac{9}{4}}^{q/2}
\end{align}
where we used the convexity of the function $s \mapsto s^{q/2}$ on $[0,\infty).$ The map $t \mapsto t^{2/q}$ is monotone increasing for any $q>0$ on $[0,\infty),$ and hence \eqref{eq:convexity} implies that the the right-hand side of \eqref{eq:mu-nu-dist} is smaller than that of \eqref{eq:Phi-mu-nu-dist}. That is, $d_{W_2}(\mu_1,\nu)<d_{W_2}(\Phi^*(\mu_1),\nu)=d_{W_2}(\Phi^*(\mu_1),\Phi^*(\nu)),$ which shows that $\Phi^*$ is not an isometry. 
\par
Note that if $q<2,$ than the inequality \eqref{eq:convexity} is reversed by the concavity of $s \mapsto s^{q/2},$ and $2^{2/q-2}>2^{-1},$ so in this case one gets $d_{W_2}(\mu_1,\nu)>d_{W_2}(\Phi^*(\mu_1),\nu)=d_{W_2}(\Phi^*(\mu_1),\Phi^*(\nu))$ which rules out $\Phi^*$ in the case $q<2$ as well.

These proofs show that, if $\Phi$ is an isometry that fixes Dirac masses  and sends measures supported on $L_1$ on measures supported on $L_1$, then $\Phi$ acts as the identity on the space $\mathcal{W}_2(L_1, |\cdot|)$, i.e. $\Phi(\mu)=\mu$ for any measure $\mu$ supported on $L_1$. Using this result, we now prove Theorem \ref{thm2}:

\begin{proof}[Proof of Thm \ref{thm2}]
We again start with the case $n=2$. By Proposition \ref{prop: restrict_lq}, we know that if a measure is supported on a translated canonical axis and $\Phi$ is an isometry of $\mathcal{W}_2(\R^2, d_q)$, then $\Phi(\nu)$ is supported on the same translated canonical axis. By the argument presented above, we can also assume that $\Phi(\nu)=\nu$ for any measure supported on $x_0 + L_i$, with $x_0 \in \R^2$, $i\in \{1, 2\}$. Since the norm $l_q$ is strictly convex, as $q>2$, it projects uniquely onto $L_1$ and $L_2$; furthermore, it is easy to verify that the projection of $L_2$ onto $L_1$ is simply $P_{L_1}(L_2)=\{0\}$. Thus we can apply Proposition \ref{prop: Hyper implies hyper plus} to prove the rigidity of the Wasserstein space $\mathcal{W}_2(\R^2, d_q)$. 

We now make the induction assumption on $n$, that for any $n \leq n_0$, the space $\mathcal{W}_2(\Rn, d_q)$ is rigid, and we consider the Wasserstein space $\mathcal{W}_p(\R^{n_0+1}, d_q)$. 

By Proposition \ref{prop: restrict_lq}, if a measure $\nu$ is supported on $x_0 + L_i$ for $x_0 \in \R^{n_0+1}, i \in \{1, \dots, n_0+1\}$ (where $L_i$ is the hyperplane given by $x_i=0$), then $\Phi(\nu)$ will also be supported on $x_0 + L_i$. Since the space $(x_0 + L_i, d_{q, x_0 + L_i})$ is isometric to $(\R^{n_0}, d_q)$, we can use the induction assumption to show that $\Phi(\nu)=\nu$ for any measures $\nu$ supported on $x_0 + L_i$.

 Since the line $H_i=\{te_i\}$ is the intersection of all hyperplanes $L_j$ with $j \neq i$, if a measure $\mu$ is supported on $H_i$, we can use Proposition \ref{prop: restrict_lq} to show that its image $\Phi(\mu)$ is also supported on every hyperplane $L_j$ with $j \neq i$. Thus its image is again supported $H_i$, and by the argument presented above, we can assume that $\Phi(\mu)=\mu$ for any measures $\mu$ supported on $H_i$. Since we still have that the $l_q$ norm projects uniquely onto $L_i$ and $H_i$, and since it is easy to see that $P_{L_i}^{-1}(0)=H_i$, we can apply Proposition \ref{prop: Hyper implies hyper plus} to prove the rigidity of the Wasserstein space $\mathcal{W}_2(\R^{n_0+1}, d_q)$. 
\end{proof}

%% file: tikztry.tex
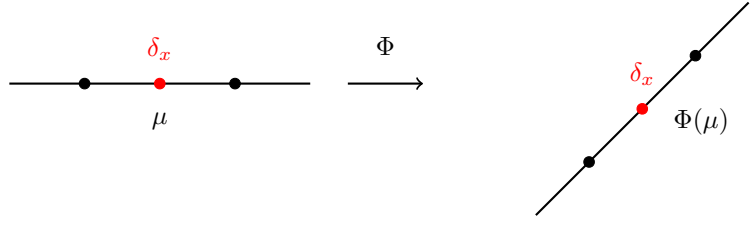
\begin{figure}
\centering

\begin{tikzpicture}[scale=1, every node/.style={font=\small}]

  \draw[thick] (0,0) -- (4,0);
  \filldraw[black] (1,0) circle (2pt) node[below=5pt] {};
  \filldraw[red] (2,0) circle (2pt) node[above=5pt] {$\delta_x$};
  \filldraw[black] (3,0) circle (2pt) node[below=5pt] {};
  \node at (2,-0.5) {$\mu$};

  \draw[->, thick] (4.5,0) -- (5.5,0);
  \node at (5, 0.5) {$\Phi$};
  
  \begin{scope}[shift={(7,-1.75)}, rotate=45]
    \draw[thick] (0,0) -- (4,0); 
    \filldraw[black] (1,0) circle (2pt) node[right=5pt] {};
    \filldraw[red] (2,0) circle (2pt) node[above=5pt] {$\delta_x$};
    \filldraw[black] (3,0) circle (2pt) node[right=5pt] {};
  \end{scope}
  \node at (9.2, -0.5) {$\Phi(\mu)$};

\end{tikzpicture}
\caption{In the Eucliden plane, there exists for $p=2$ shape-preserving isometries of $\mathcal{W}_2(\R^2, d_E)$ that send a measure supported on a line to a measure supported on a different line.}
\label{fig:M3}
\end{figure}

%% file: main.bbl
\begin{thebibliography}{99}


\bibitem{AG}
\textsc{L. Ambrosio, N. Gigli}, 
\textit{A user’s guide to optimal transport}. In: Modelling and optimisation of flows on networks, Lecture Notes in Math., 2062, Springer Heidelberg, (2013).


\bibitem{BKTV} 
\textsc{Z. M. Balogh, G. Kiss, T. Titkos, D. Virosztek}, \textit{Isometric rigidity of the Wasserstein space over the plane with the maximum metric}, Can. J. Math, Published online:1-30 (2025).

\bibitem{BTV} 
\textsc{Z. M. Balogh, T. Titkos, D. Virosztek}, \textit{Isometries and isometric embeddings of Wasserstein spaces over the Heisenberg group}, Rev. Mat. Iberoam. 41 (2025), 2055-2084. 

\bibitem{BK} 
\textsc{J. Bertrand, B. Kloeckner}, \textit{A geometric study of Wasserstein spaces: isometric rigidity in negative curvature},
Int. Math. Res. Notices,  5, (2016), 1368--1386.

\bibitem{BK2}
\textsc{J. Bertrand, B. Kloeckner}, \textit{A geometric study of Wasserstein spaces: Hadamard spaces}, Journal of Topology and Analysis, Vol. 04 (2012), 515--542.

\bibitem{S-R2}
\textsc{M. Che, F. Galaz-García, M. Kerin, J. Santos-Rodríguez}, \textit{Isometric Rigidity of Metric Constructions with respect to Wasserstein Spaces}, manuscript, arXiv:2410.14648, (2024).

\bibitem{Figalli}
\textsc{A. Figalli, F. Glaudo},
\textit{An Invitation to Optimal Transport, Wasserstein Distances, and Gradient Flows}, EMS Textbooks in Mathematics, Volume 23, (2021).

\bibitem{fletcher_chebyshev_2015}
\textsc{J. Fletcher, W. B. Moors}, 
\textit{Chebychev Sets}, J. Aust. Math. Soc. 98, (2015): 161–231. 

\bibitem{GTV1} \textsc{Gy. P. Geh\'er, T. Titkos, D. Virosztek}, \textit{Isometric study of Wasserstein spaces-the real line}, Transactions of the American Math. Society, 373, (2020), 5855-5883.
			
\bibitem{GTV2} 
\textsc{Gy. P. Geh\'er, T. Titkos, D. Virosztek}, \textit{Isometry group of Wasserstein spaces: the Hilbertian case}, Journal of the London Math. Soc., 106, (2022), 3865--3894.

\bibitem{K} 
\textsc{B. Kloeckner},  \textit{A geometric study of Wasserstein spaces: Euclidean spaces},  Ann. Sc. Norm. Super. Pisa Cl. Sci., 5, (2010), 297--323. 

\bibitem{Pap}
\textsc{A. Papadopoulos}, \textit{Metric spaces, convexity and nonpositive curvature}, EMS IRMA Lectures, (2005).

\bibitem{S-R}
\textsc{J. Santos-Rodríguez}, 
\textit{On isometries of compact $L^p$–Wasserstein spaces}, Adv. Math., {409}, (2022), Article No. 108632.

\bibitem{STV} 
\textsc{E. Str\"oher, T. Titkos, D. Virosztek}, \textit{Wasserstein Rigidity for $l_q$ norms}, In preparation. 

\bibitem{V} 
\textsc{C. Villani}, \textit{Topics in Optimal Transportation}, Graduate Studies in Mathematics, volume 58, Amer. Math. Soc. Providence, RI, (2003). 		
		
\bibitem{Villani} 
\textsc{C. Villani}, \textit{Optimal Transport: Old and New},
Grundlehren der mathematischen Wissenschaften, 338, Springer Heidelberg, (2009).

\end{thebibliography}
